\documentclass[a4paper,11pt]{article}

\usepackage[T1]{fontenc}
\usepackage{lmodern}

\usepackage{dsfont}

\usepackage[latin1]{inputenc}
\usepackage{amsmath}
\usepackage{amsthm}
\usepackage{amssymb}
\usepackage{mathrsfs}
\usepackage{graphicx}
\usepackage[all]{xy}
\usepackage{hyperref}

\usepackage{stmaryrd}
\usepackage{caption}

\usepackage{abstract} 

\newtheorem{thm}{Theorem}[section]
\newtheorem{cor}[thm]{Corollary}
\newtheorem{claim}[thm]{Claim}
\newtheorem{fact}[thm]{Fact}

\newtheorem{lemma}[thm]{Lemma}
\newtheorem{prop}[thm]{Proposition}

\theoremstyle{definition}
\newtheorem{definition}[thm]{Definition}
\newtheorem{ex}[thm]{Example}
\newtheorem{remark}[thm]{Remark}
\newtheorem{question}[thm]{Question}
\newtheorem{problem}[thm]{Problem}

\title{Quasi-isometrically rigid subgroups in right-angled Coxeter groups}
\date{\today}
\author{Anthony Genevois}

\begin{document}

\maketitle

\begin{abstract}
In the spirit of peripheral subgroups in relatively hyperbolic groups, we exhibit a simple class of quasi-isometrically rigid subgroups in graph products of finite groups, which we call \emph{eccentric subgroups}. As an application, we prove that, if two right-angled Coxeter groups $C(\Gamma_1)$ and $C(\Gamma_2)$ are quasi-isometric, then for any minsquare subgraph $\Lambda_1 \leq \Gamma_1$ there exists a minsquare subgraph $\Lambda_2 \leq \Gamma_2$ such that the right-angled Coxeter groups $C(\Lambda_1)$ and $C(\Lambda_2)$ are quasi-isometric as well. Various examples of non-quasi-isometric groups are deduced. Our arguments are based on a study of non-hyperbolic Morse subgroups in graph products of finite groups. As a by-product, we are able to determine precisely when a right-angled Coxeter group has all its infinite-index Morse subgroups hyperbolic, answering a question of Russell, Spriano and Tran.
\end{abstract}

\tableofcontents

\newpage

\section{Introduction}

\noindent
In this article, we are interested in the large-scale geometry of graph product of finite groups. More precisely, if $\Gamma$ is a simplicial graph and $\mathcal{G}= \{ G_u \mid u \in V(\Gamma)\}$ a collection of (finite) groups indexed by the vertex-set $V(\Gamma)$ of $\Gamma$, then the \emph{graph product} $\Gamma \mathcal{G}$ is defined as the quotient
$$\left( \underset{u \in V(\Gamma)}{\ast} G_u \right) / \langle \langle [g,h]=1, g \in G_u, h \in G_v \ \text{if} \ (u,v) \in E(\Gamma) \rangle \rangle$$
where $E(\Gamma)$ denotes the edge-set of $\Gamma$. For instance, $\Gamma \mathcal{G}$ is the direct sum of $\mathcal{G}$ if $\Gamma$ is a complete graph, and the free product of $\mathcal{G}$ if $\Gamma$ has no edges. Thus, graph products define an interpolation between free products and direct sums of groups. Now the question is, given two finite simplicial graphs $\Gamma_1, \Gamma_2$ and two collections of finite groups $\mathcal{G}_1, \mathcal{G}_2$ indexed by $V(\Gamma_1), V(\Gamma_2)$ respectively, to determine whether the groups $\Gamma_1 \mathcal{G}_1$ and $\Gamma_2 \mathcal{G}_2$ are quasi-isometric. For instance, the seven graph products given by Figure \ref{Intro} are pairwise non-quasi-isometric. (It is not difficult to show that the quasi-isometry class of a graph product of finite groups depends only on the cardinality of the vertex-groups, see for instance \cite[Fact 8.25]{Qm}. Therefore, we labelled the vertices of the graphs of Figure \ref{Intro} by the cardinalities of the corresponding vertex-groups.)

\medskip \noindent
In the specific case where all our finite groups are cyclic of order two, the problem we are interested in amounts to classifying right-angled Coxeter groups up to quasi-isometry. We refer to the survey \cite{QIandRACG} and references therein for more information on this well-known and difficult problem. More general graph products of finite groups have been less studied, but the study of right-angled buildings led to several interesting results, including M. Bourdon's seminal work \cite{Bourdon}. 

\medskip \noindent
In this article, we exhibit a surprising rigidity phenomenon by showing that a quasi-isometry between two graph products of finite groups always preserves a simple class of subgroups. As these subgroups turn out to be (smaller) graph products of finite groups, such a rigidity may allow us to reduce the complexity of a quasi-isometry problem. 

\medskip \noindent
But before stating our theorem, we need the following definition. Given a simplicial graph $\Gamma$, a subgraph $\Lambda \leq \Gamma$ is \emph{square-complete} if every induced square of $\Gamma$ containing two opposite vertices in $\Lambda$ must be entirely included into $\Lambda$. A \emph{minsquare subgraph} of $\Gamma$ is a subgraph which is minimal among all the square-complete subgraphs of $\Gamma$ containing at least one induced square. Now, the main result of our article is:

\begin{thm}\label{thm:IntroMain}
Let $\Gamma_1,\Gamma_2$ be two finite simplicial graphs and $\mathcal{G}_1, \mathcal{G}_2$ two collections of finite groups indexed by $V(\Gamma_1), V(\Gamma_2)$ respectively. Assume that there exists a quasi-isometry $\Phi : \Gamma_1 \mathcal{G}_1 \to \Gamma_2 \mathcal{G}_2$. For every minsquare subgraph $\Lambda_1 \subset \Gamma_1$, there exist an element $g \in \Gamma_2 \mathcal{G}_2$ and a minsquare subgraph $\Lambda_2 \subset \Gamma_2$ such that $\Phi$ sends $\langle \Lambda_1 \rangle$ at finite Hausdorff distance from $g \langle \Lambda_2 \rangle$, where $\langle \Lambda_1 \rangle, \langle \Lambda_2 \rangle$ denote the subgroups generated by the groups labelling the vertices of $\Lambda_1, \Lambda_2$ respectively. 
\end{thm}

\noindent
We emphasize that the subgroups $\langle \Lambda_1 \rangle$ and $\langle \Lambda_2 \rangle$ are naturally graph products of groups themselves. Therefore, the philosophy behind Theorem \ref{thm:IntroMain} is that we deduce from a quasi-isometry problem between two graph products of finite groups a quasi-isometry problem between two smaller (and hopefully simpler) graph products of finite groups. 

\medskip \noindent
Theorem \ref{thm:IntroMain} should be compared with the following statement, which is an easy combination of the quasi-isometric rigidity of peripheral subgroups in relatively hyperbolic groups \cite{BDM} and the description of a minimal collection of peripheral subgroups in graph products of finite groups \cite{Qm} (see \cite{BHSC} or \cite{coningoff} for the particular case of right-angled Coxeter groups).

\begin{thm}\label{thm:IntroRH}
Let $\Gamma_1,\Gamma_2$ be two finite simplicial graphs and $\mathcal{G}_1, \mathcal{G}_2$ two collections of finite groups indexed by $V(\Gamma_1), V(\Gamma_2)$ respectively. Assume that there exists a quasi-isometry $\Phi : \Gamma_1 \mathcal{G}_1 \to \Gamma_2 \mathcal{G}_2$. For every subgraph $\Lambda_1 \in \mathfrak{J}^\infty( \Gamma_1)$, there exist an element $g \in \Gamma_2 \mathcal{G}_2$ and a subgraph $\Lambda_2 \in \mathfrak{J}^\infty( \Gamma_2)$ such that $\Phi$ sends $\langle \Lambda_1 \rangle$ at finite Hausdorff distance from $g \langle \Lambda_2 \rangle$, where $\langle \Lambda_1 \rangle, \langle \Lambda_2 \rangle$ denote the subgroups generated by the groups labelling the vertices of $\Lambda_1, \Lambda_2$ respectively. 
\end{thm}

\noindent
Given a simplicial graph $\Gamma$, $\mathfrak{J}^\infty(\Gamma)$ is a collection of subgraphs of $\Gamma$ which encodes the relative hyperbolicity of any graph product of finite groups defined over $\Gamma$. We refer to Section \ref{section:RH} for a precise definition.  

\medskip \noindent
Up to our knowledge, Theorems \ref{thm:IntroMain} and \ref{thm:IntroRH} provide the only two known classes of quasi-isometrically rigid subgroups in graph products of finite groups (or even in right-angled Coxeter groups).
\begin{figure}
\begin{center}
\includegraphics[scale=0.35]{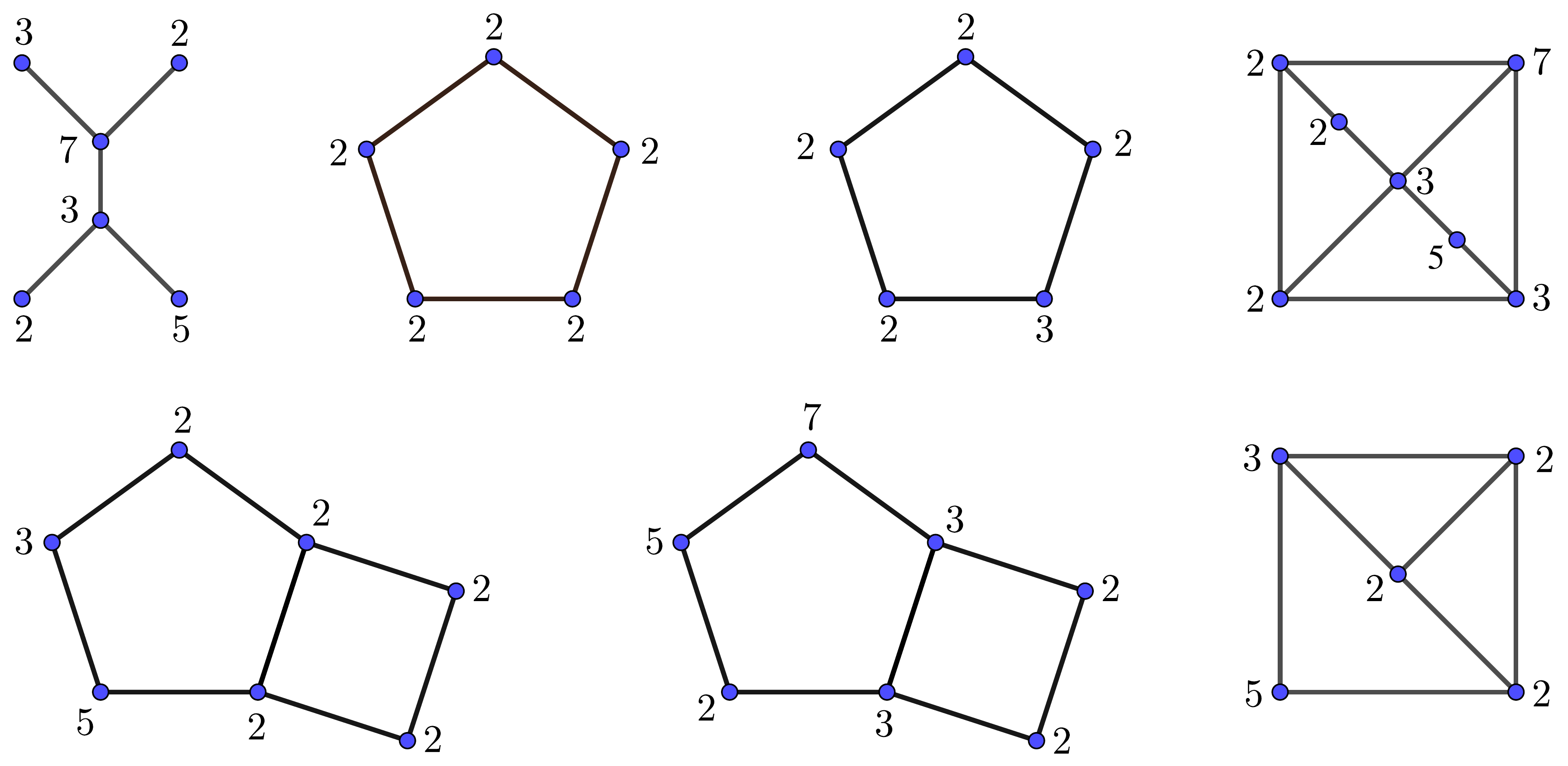}
\caption{The first graph product is virtually free \cite{VirtuallyFreeGP}; the second one is virtually a surface group; the third one is hyperbolic with a boundary containing infinitely many circles; the fourth one is not relatively hyperbolic and has superlinear divergence \cite{MinasyanOsin,Sistohypembed}; the fifth one is toral relatively hyperbolic; the sixth one is hyperbolic relative to a subgroup virtually $\mathbb{F}_2 \times \mathbb{F}_2$; the seventh one is quasi-isometric to $\mathbb{F}_2 \times \mathbb{F}_2$.}
\label{Intro}
\end{center}
\end{figure}

\medskip \noindent
Interestingly, Theorem \ref{thm:IntroMain} allows us to construct quasi-isometric invariants. For instance, it can be proved that:

\begin{prop}\label{thm:IntroQIinv}
Let $\Gamma_1,\Gamma_2$ be two finite simplicial graphs and $\mathcal{G}_1, \mathcal{G}_2$ two collections of finite groups indexed by $V(\Gamma_1), V(\Gamma_2)$ respectively. Assume that the graph products $\Gamma_1 \mathcal{G}_1$ and $\Gamma_2 \mathcal{G}_2$ are quasi-isometric. 
\begin{itemize}
	\item If $\Gamma_1$ is a minsquare graph, then $\Gamma_2$ decomposes as the join of a minsquare subgraph and a complete graph.
	\item If $\Gamma_1$ contains an induced square whose vertices are labelled by $\mathbb{Z}/2\mathbb{Z}$'s  and which is square-complete, then so does $\Gamma_2$. 
\end{itemize}
\end{prop}

\noindent
From another point of view, given a simplicial graph $\Gamma$ and a collection of groups $\mathcal{G}$ indexed by $V(\Gamma)$, we define the \emph{electrification} $\mathbb{E}(\Gamma, \mathcal{G})$ as the Cayley graph of $\Gamma \mathcal{G}$ constructed from the generating set which is the union of the groups of $\mathcal{G}$ and the subgroups $\langle \Lambda \rangle$ where $\Lambda$ is a minsquare subgraph of $\Gamma$. As a consequence of Theorem \ref{thm:IntroMain}, the electrification defines a quasi-isometric invariant. More precisely:

\begin{prop}
Let $\Gamma_1,\Gamma_2$ be two simplicial graphs and $\mathcal{G}_1, \mathcal{G}_2$ two collections of finite groups indexed by $V(\Gamma_1), V(\Gamma_2)$ respectively. Any quasi-isometry $\Gamma_1 \mathcal{G}_1 \to \Gamma_2 \mathcal{G}_2$ induces a quasi-isometry $\mathbb{E}(\Gamma_1,\mathcal{G}_1) \to \mathbb{E}(\Gamma_2,\mathcal{G}_2)$.
\end{prop}

\noindent
Consequently, studying the large-scale geometry of the electrification provides other quasi-isometric invariants. Our main result in this direction is the following characterisation of the hyperbolicity of the electrification:

\begin{prop}\label{prop:IntroElecHyp}
Let $\Gamma$ be a finite simplicial graph and $\mathcal{G}$ a collection of groups indexed by $V(\Gamma)$. The electrification $\mathbb{E}(\Gamma, \mathcal{G})$ is hyperbolic if and only if every induced square of $\Gamma$ is included into some minsquare subgraph. 
\end{prop}

\noindent
In Section \ref{section:ex}, various examples of right-angled Coxeter groups are showed not to be quasi-isometric by applying the quasi-isometric invariants discussed above.

\paragraph{Morse subgroups and their geometries.}
Theorem \ref{thm:IntroMain} follows from a study of non-hyperbolic \emph{Morse subgroups} in graph products of finite groups. 

\medskip \noindent
The good properties satisfied by quasiconvex subgroups in hyperbolic groups motivated several definitions of subgroups embedded in some ``hyperbolic'' way, including relatively quasiconvex subgroups in relatively hyperbolic groups (see the survey \cite{HruskaRH} and references therein) and hyperbolically embedded subgroups \cite{DGO} (leading to the theory of acylindrically hyperbolic groups \cite{OsinAcyl}). The two most recent definitions of such subgroups are \emph{stable subgroups} \cite{StableSub} and \emph{Morse subgroups} \cite{StronglyQC, MoiHypCube}, the latter being a generalisation of the former.

\begin{definition}
Let $X$ be a geodesic metric space. Then $Y \subset X$ is a \emph{Morse subspace} if, for every $A>0$ and $B \geq 0$, there exists some $M(A,B) \geq 0$ such that any $(A,B)$-quasigeodesic between two points of $Y$ lies in the $M(A,B)$-neighborhood of $Y$. The map $M$ is referred to as the \emph{Morse gauge} of $Y$. If $X$ is the Cayley graph of a group $G$ constructed from a finite generating set and $Y$ a subgroup $H$ of $G$, one says that $H$ is a \emph{Morse subgroup} of $G$; and if $H= \langle g \rangle$, one says that $g$ is a Morse element of~$G$. 
\end{definition}

\noindent
Although Morse elements have been studied for several decades, Morse subgroups were introduced only very recently as a notion of independent interest. First appearing implicitly in \cite{Sistohypembed}, Morse subgroups are introduced independently in \cite{StronglyQC} (as \emph{strongly quasiconvex subgroups}) and in \cite{MoiHypCube}. In \cite{MoiHypCube}, they appear as a very convenient tool in the study of relative hyperbolicity of groups acting on CAT(0) cube complexes; and, in \cite{StronglyQC}, it is shown that they satisfy most of the good properties holding for quasiconvex subgroups in hyperbolic groups. 

\medskip \noindent
In the article, we are interested in specific Morse subspaces, and the goal will be to show that, in graph products of finite groups, they are at finite Hausdorff distance from subgroups.

\begin{definition}
Let $X$ be a geodesic metric space. A subspace $Y \subset X$ is \emph{eccentric} if it is Morse, non-hyperbolic, and if, for every map $m : (0,+ \infty) \times [0,+ \infty) \to \mathbb{R}$, the Hausdorff distance between $Y$ and any non-hyperbolic Morse subspace $Z \subset Y$ with Morse-gauge $m$ is bounded above by a finite constant $E(m)$. The map $E$ is referred to as the \emph{eccentric-gauge} of $Y$. 
\end{definition}

\noindent
Now, Theorem \ref{thm:IntroMain} is an easy consequence of the following statement:

\begin{thm}\label{thm:IntroEccentric}
Let $\Gamma$ be a finite simplicial graph and $\mathcal{G}$ a collection of finite groups indexed by $V(\Gamma)$. A subspace $M \subset \Gamma \mathcal{G}$ is eccentric if and only if there exists a minsquare subgraph $\Lambda \subset \Gamma$ such that $M$ is at finite Hausdorff distance from a coset of $\langle \Lambda \rangle$.
\end{thm}

\noindent
The theorem motivates the following definition. Given a simplicial graph $\Gamma$ and a collection of finite groups $\mathcal{G}$ indexed by $V(\Gamma)$, a subgroup $H \leq \Gamma \mathcal{G}$ is \emph{eccentric} if $H= g \langle \Lambda \rangle g^{-1}$ for some element $g \in \Gamma \mathcal{G}$ and some minsquare subgraph $\Lambda \leq \Gamma$.  

\medskip \noindent
As a by-product of our study of eccentric subgroups, we are able to answer a question of \cite{ConvexityHHS}. Although right-angled Artin groups contain few Morse subgroups, as infinite-index Morse subgroups in freely irreducible right-angled Artin groups turn out to be free \cite{StronglyQC, MoiHypCube} (see also Remark \ref{remark:MorseRAAG} below for an alternative argument), right-angled Coxeter groups may contain various Morse subgroups (see for instance \cite[Theorem~F]{ConvexityHHS}). A sufficient but not necessary condition for the infinite-index Morse subgroups of a right-angled Coxeter group to be all hyperbolic is given in \cite{ConvexityHHS}, and \cite[Question 2]{ConvexityHHS} naturally asks for a necessary and sufficient condition. Our next statement answers this question.

\begin{thm}\label{thm:MorseSubHyp}
Let $\Gamma$ be a simplicial graph and $\mathcal{G}$ a collection of finite groups indexed by $V(\Gamma)$. The infinite-index Morse subgroups of $\Gamma \mathcal{G}$ are all hyperbolic if and only if $\Gamma$ is square-free or if it decomposes as the join of a minsquare subgraph and a complete graph.  
\end{thm}

\noindent
A natural auxiliary problem would to determine when the infinite-index Morse subgroups are all free. Unfortunately, we were not able to answer this question. See Question \ref{question:MorseFree}.

\paragraph{Organisation of the article.} Section \ref{section:GPandQM} is essentially a discussion about quasi-median graphs, which are our geometric models when dealing with graph products of finite groups. The first two subsections mainly come from \cite{Qm}, except for a few preliminary lemmas; and the third subsection is an adaptation of arguments of \cite{coningoff} for quasi-median graphs. The most original part of the article is Section~\ref{section:eccentric}, which contains the proof of Theorem \ref{thm:IntroEccentric}. Theorem \ref{thm:IntroRH} is proved in Section \ref{section:RH} by combining two results of \cite{BDM} and \cite{Qm}. Next, the electrification is studied in Section \ref{section:elec}. More precisely, its invariance under quasi-isometries (which is almost an immediate consequence of Theorem \ref{thm:IntroEccentric}) is proved in Subsection \ref{subsection:elec}; and we characterise its hyperbolicity in Subsection \ref{subsection:elechyp} by proving Proposition \ref{prop:IntroElecHyp} thanks to methods from \cite{coningoff}. The main theoretical application of our work, namely Theorem \ref{thm:MorseSubHyp}, is proved in Section \ref{section:MorseSubHyp}, answering a question of Russell, Spriano and Tran. We conclude the article by a few explicit examples in Section \ref{section:ex} and a few open questions in Section \ref{section:questions}.

\paragraph{Acknowledgments.} This work was supported by a public grant as part of the Fondation Math\'ematique Jacques Hadamard.

\section{Graph products and quasi-median geometry}\label{section:GPandQM}

\noindent
In order to study the large-scale geometry of graph products of finite groups, we will exploit the geometric model introduced in \cite{Qm}. More precisely, the Cayley graph of a graph product of finite groups with the union of the vertex-groups as a generating set turns out to be a \emph{quasi-median graph}, whose geometry is very close to the geometry of CAT(0) cube complexes, generalising the well-known fact that the Cayley graph of a right-angled Coxeter group with respect to its canonical generating set defines the one-skeleton of a CAT(0) cube complex. In Section \ref{section:QM}, we give basic definitions and properties related to quasi-median graphs; and in Section \ref{section:GPandQM}, we describe quasi-median graphs associated to graph products. Finally, in Section \ref{section:gated}, we focus on Morse subgraphs in quasi-median graphs, proving a few preliminary statements which will be fundamental in the sequel.

\subsection{Quasi-median graphs}\label{section:QM}

\noindent
There exist several equivalent definitions of quasi-median graphs, see for instance \cite{quasimedian}. Below is the definition used in \cite{Qm}.

\begin{definition}
A graph $X$ is \emph{quasi-median} if it does not contain $K_4^-$ and $K_{3,2}$ as induced subgraphs, and if it satisfies the following two conditions:
\begin{description}
	\item[(triangle condition)] for every vertices $a, x,y \in X$, if $x$ and $y$ are adjacent and if $d(a,x)=d(a,y)$, then there exists a vertex $z \in X$ which adjacent to both $x$ and $y$ and which satisfies $d(a,z)=d(a,x)-1$;
	\item[(quadrangle condition)] for every vertices $a,x,y,z \in X$, if $z$ is adjacent to both $x$ and $y$ and if $d(a,x)=d(a,y)=d(a,z)-1$, then there exists a vertex $w \in X$ which adjacent to both $x$ and $y$ and which satisfies $d(a,w)=d(a,z)-2$.
\end{description}
\end{definition}

\noindent
The graph $K_{3,2}$ is the bipartite complete graph, corresponding to two squares glued along two adjacent edges; and $K_4^-$ is the complete graph on four vertices minus an edge, corresponding to two triangles glued along an edge. The triangle and quadrangle conditions are illustrated by Figure \ref{Quadrangle}.
\begin{figure}
\begin{center}
\includegraphics[scale=0.45]{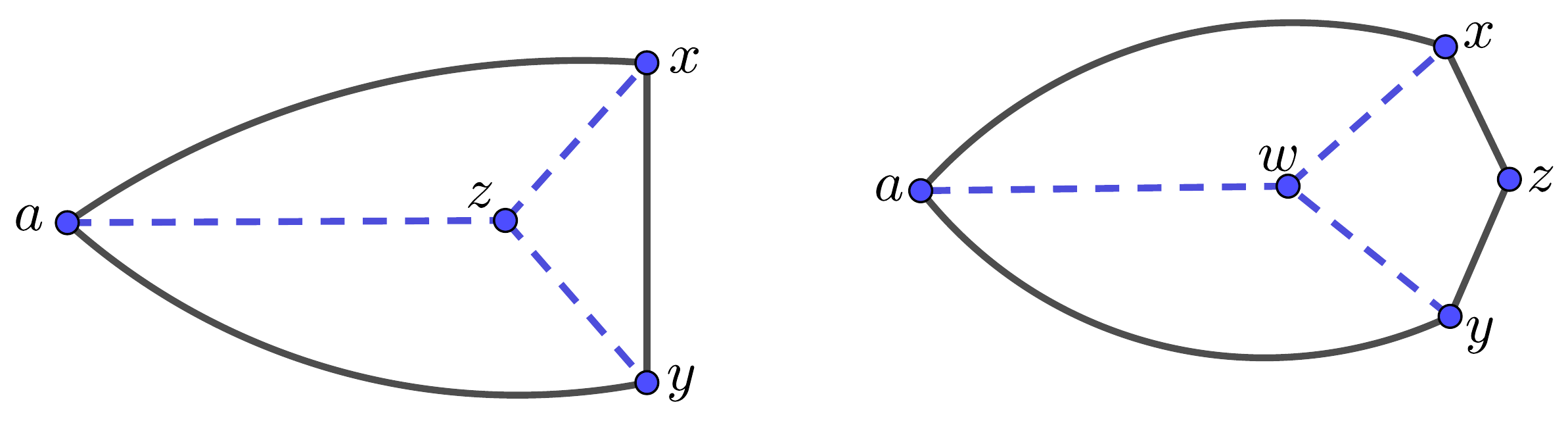}
\caption{Triangle and quadrangle conditions.}
\label{Quadrangle}
\end{center}
\end{figure}

\begin{definition}
Let $X$ be a graph and $Y \subset X$ a subgraph. A vertex $y \in Y$ is a \emph{gate} of an other vertex $x \in X$ if, for every $z \in Y$, there exists a geodesic between $x$ and $z$ passing through $y$. If every vertex of $X$ admits a gate in $Y$, then $Y$ is \emph{gated}.
\end{definition}

\noindent
It is worth noticing that the gate of $x$ in $Y$, when it exists, is unique and minimises the distance to $x$ in $Y$. Gated subgraphs in quasi-median graphs play the role of convex subcomplexes in CAT(0) cube complexes.

\begin{lemma}\label{lem:GatedCriterion}\emph{\cite{Chepoi}}
Let $X$ be a quasi-median graph and $Y \subset X$ a connected induced subgraph. Then $Y$ is gated if and only if it is \emph{convex} (i.e., any geodesic between two vertices of $Y$ lies in $Y$) and it \emph{contains its triangles} (i.e., any triangle having an edge in $Y$ lies entirely in $Y$). 
\end{lemma}

\noindent
Quasi-median graphs can be naturally thought of as prism complexes, but the notion of dimension which interests us is not the usual dimension of this prism complex but the following:

\begin{definition}
Let $X$ be a graph. Its \emph{cubical dimension}, denoted by $\dim_\square(X)$, is the maximal dimension of a cube whose one-skeleton embeds as an induced subgraph into $X$.
\end{definition}

\noindent
Similarly to CAT(0) cube complexes, the cubical dimension turns out to coincides with the maximal cardinality of a collection of pairwise transverse hyperplanes:

\begin{lemma}\emph{\cite[Proposition 2.73]{Qm}}
Let $X$ be a quasi-median graph. The cubical dimension of $X$ coincides with the maximal cardinality of a collection of pairwise transverse hyperplanes.
\end{lemma}

\noindent
Finally, recall that a \emph{clique} is a maximal complete subgraph, and that cliques in quasi-median graphs are gated \cite{quasimedian}.

\paragraph{Hyperplanes.} Similarly to CAT(0) cube complexes, the notion of \emph{hyperplane} is fundamental in the study of quasi-median graphs.

\begin{definition}
Let $X$ be a graph. A \emph{hyperplane} $J$ is an equivalence class of edges with respect to the transitive closure of the relation saying that two edges are equivalent whenever they belong to a common triangle or are opposite sides of a square. We denote by $X \backslash \backslash J$ the graph obtained from $X$ by removing the interiors of all the edges of $J$. A connected component of $X \backslash \backslash J$ is a \emph{sector}. The \emph{neighborhood} of $J$, denoted by $N(J)$, is the subgraph generated by all the edges of $J$. Two hyperplanes $J_1$ and $J_2$ are \emph{transverse} if there exist two edges $e_1 \subset J_1$ and $e_2 \subset J_2$ spanning a square in $X$; and they are \emph{tangent} if they are not transverse but $N(J_1) \cap N(J_2) \neq \emptyset$. 
\end{definition}
\begin{figure}
\begin{center}
\includegraphics[trim={0 16.5cm 10cm 0},clip,scale=0.45]{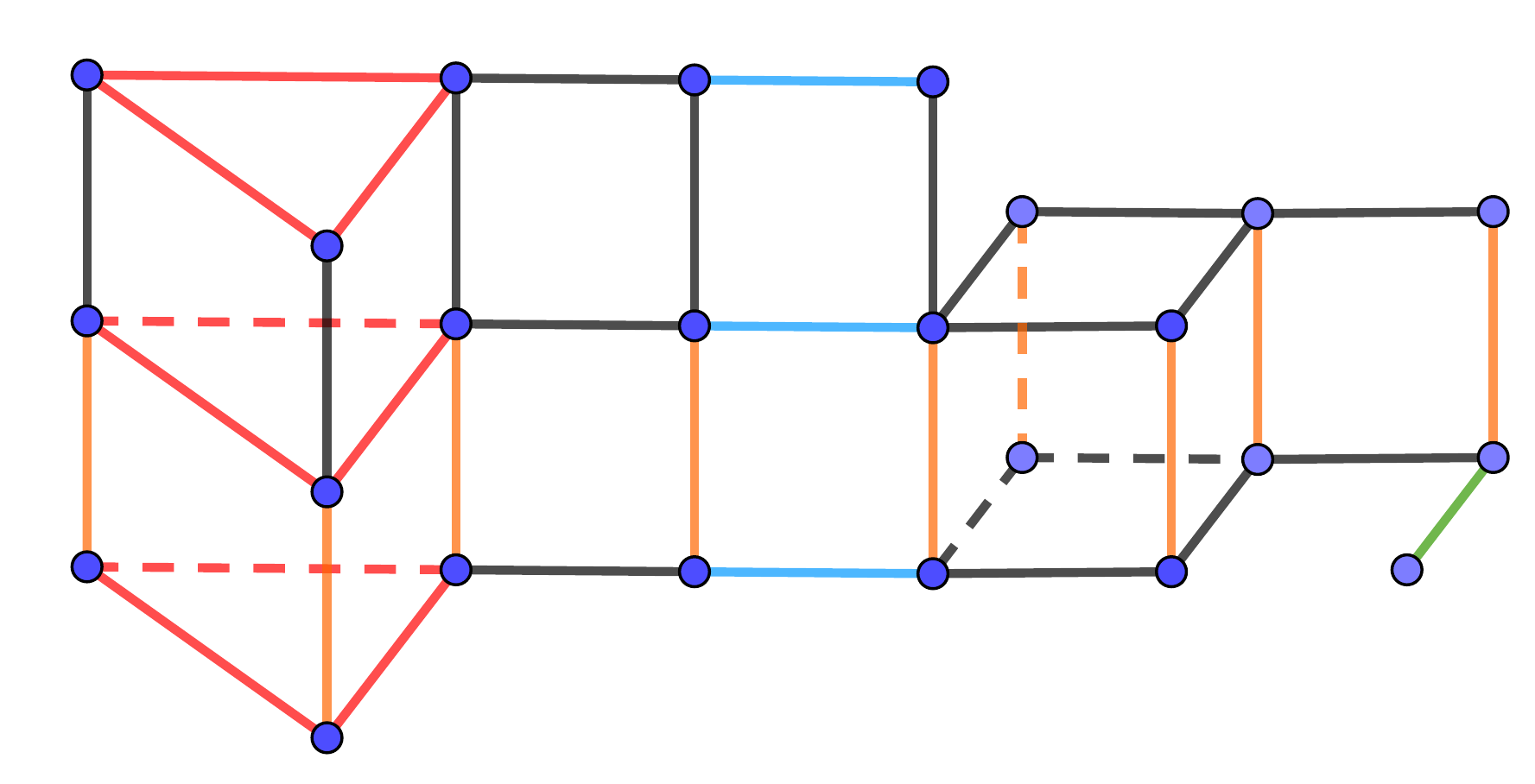}
\caption{A quasi-median graph and some of its hyperplanes.}
\label{figure3}
\end{center}
\end{figure}

\noindent
See Figure \ref{figure3} for examples of hyperplanes in a quasi-median graph. The connection between hyperplanes and geometry is made explicit by the following two theorems:

\begin{thm}\label{thm:BigThmQM}\emph{\cite[Proposition 2.15]{Qm}}
Let $X$ be a quasi-median graph and $J$ a hyperplane. The graph $X \backslash \backslash J$ is disconnected, and the neighborhood and the sectors of $J$ are gated. 
\end{thm}

\begin{thm}\label{thm:GeodesicsQM}\emph{\cite[Proposition 2.30]{Qm}}
Let $X$ be a quasi-median graph and $\gamma$ a path between two vertices $x,y \in X$. The following assertions are equivalent:
\begin{itemize}
	\item $\gamma$ is a geodesic;
	\item $\gamma$ crosses a hyperplane if and only if it separates $x$ and $y$;
	\item $\gamma$ crosses each hyperplane at most once.
\end{itemize}
As a consequence, the distance between $x$ and $y$ coincides with the number of hyperplanes separating them.
\end{thm}

\paragraph{Projections onto gated subgraphs.} As mentioned earlier, the gate of a vertex $x$ in a gated subgraph $Y$ coincides with the unique vertex of $Y$ minimising the distance to $x$. From now on, we refer to the gate of $x$ in $Y$ as the \emph{projection} of $x$ onto $Y$. 

\begin{prop}\label{prop:BigProj}
Let $X$ be a quasi-median graph and $Y,Y_1,Y_2 \subset X$ three gated subgraphs. The following assertions hold:
\begin{itemize}
	\item[(i)] \emph{\cite[Lemma 2.34]{Qm}} A hyperplane separating a vertex $x$ from its projection onto $Y$ separates $x$ from $Y$. 
	\item[(ii)] \emph{\cite[Lemma 2.36]{Qm}} If $x_1 \in Y_1$ and $x_2 \in Y_2$ are two vertices minimising the distance between $Y_1$ and $Y_2$ then the hyperplanes separating $x_1$ and $x_2$ are exactly the hyperplanes separating $Y_1$ and $Y_2$.
	\item[(iii)] \emph{\cite[Proposition 2.33]{Qm}} The hyperplanes separating the projections onto $Y$ of two vertices $x,y \in X$ are exactly the hyperplanes separating $x$ and $y$ which cross $Y$. As a consequence, the projection onto $Y$ is $1$-Lipschitz.
\end{itemize}
\end{prop}

\noindent
The following lemma will be also needed in the sequel:

\begin{lemma}\label{lem:GatedBridge}
Let $X$ be a quasi-median graph and $Y,Z \subset X$ two gated subgraphs. Let $M \subset Y$ denote the subgraph generated by the vertices $y \in Y$ satisfying $d(y,Z)=d(Y,Z)$. Then $M$ is a gated subgraph, and the hyperplanes crossing $M$ are exactly the hyperplanes crossing both $Y$ and $Z$.
\end{lemma}

\begin{proof}
For convenience, we denote by $p : X \to Z$ the projection onto $Z$.

\medskip \noindent
Let $J$ be a hyperplane separating two vertices $x,y \in M$. If $J$ does not separate $p(x)$ and $p(y)$, then it has to separate $x$ and $p(x)$ or $y$ and $p(y)$. But, according to Proposition~\ref{prop:BigProj}, any hyperplane separating $x$ and $p(x)$ or $y$ and $p(y)$ must separate $Y$ and $Z$, which is not the case of $J$ since it crosses $M \subset X$. Therefore, we have proved:

\begin{fact}\label{fact:HypProj}
A hyperplane separating two vertices of $M$ also separates their projections onto $Z$. 
\end{fact}

\noindent
Now, we claim that $Z$ is convex. So let $x,y,z \in M$ be three vertices such that $z$ belongs to a geodesic between $x$ and $y$. Let $J$ be a hyperplane separating $z$ and $p(z)$. According to Proposition \ref{prop:BigProj}, $J$ does not cross $Z$, so that it follows from Fact \ref{fact:HypProj} that $J$ does not separate $x$ from $z$ nor $z$ from $y$. Consequently, $J$ separates $\{x,y,z\}$ and $\{p(x),p(z),p(y)\}$. In particular, $J$ separates $x$ and $p(x)$, so we deduce from Proposition~\ref{prop:BigProj} that $J$ separates $Y$ and $Z$. Thus, we have proved that any hyperplane separating $z$ from $p(z)$ has to separate $Y$ and $Z$. Conversely, any hyperplane separating $Y$ and $Z$ clearly separates $z$ and $p(z)$. Therefore, the distance between $z$ and $p(z)$ must be equal to $d(Y,Z)$, hence $z \in M$, concluding the proof of our claim. 

\medskip \noindent
According to Lemma \ref{lem:GatedCriterion}, in order to show that $M$ is gated, it is sufficient to show that $M$ contains its triangles. So let $a,b,c \in X$ be three pairwise adjacent vertices with $b,c \in M$. Notice that, as $Y$ is a gated subgraph containing $M$, necessarily $a \in Y$. Consequently, $d(a,p(b)) \geq d(Y,Z)$. If we have equality, then $a \in M$ and we are done, so we suppose that $d(a,p(b))> d(Y,Z)$. Because $d(a,p(b)) \leq d(a,b)+d(b,p(b)) = 1+d(Y,Z)$, we must have $d(a,p(b))=d(Y,Z)+1$. Similarly, we may suppose without loss of generality that $d(a,p(c))=d(Y,Z)+1$. Next, notice that $p(b)$ and $p(c)$ are adjacent vertices. Indeed, $d(p(b),p(c)) \leq d(b,c) =1$ according to Proposition \ref{prop:BigProj}; and it follows from Fact \ref{fact:HypProj} that the hyperplane separating $b$ and $c$ also separates $p(b)$ and $p(c)$. Therefore, we can apply the triangle condition to $\{a,p(b),p(c)\}$, and we find that there exists a vertex $p \in X$ which is adjacent to both $p(b)$ and $p(c)$ and such that $d(a,p)=d(a,p(b))-1= d(Y,Z)$. Notice that $p$ belongs to $Z$, since the fact that $Z$ is gated implies that it contains it triangles. We conclude that $a \in M$ as desired.

\medskip \noindent
Thus, we have proved the first assertion of our lemma. It remains to show that the hyperplanes crossing $M$ are exactly the hyperplanes crossing both $Y$ and $Z$. We already know from Fact \ref{fact:HypProj} that a hyperplane crossing $M$ has to cross both $Y$ and $Z$. So let $J$ be a hyperplane crossing both $Y$ and $Z$. 

\medskip \noindent
Fix two vertices $y \in Y$ and $z \in Z$ minimising the distance between $Y$ and $Z$, and let $S$ denote a sector delimited by $J$ which does not contain $y$ and $z$ (notice that $J$ cannot separate $y$ and $z$ because any such hyperplane has to separate $Y$ and $Z$ according to Proposition \ref{prop:BigProj}). Let $x',y'$ denote respectively the projections of $x,y$ onto $S$. By construction, $J$ separates $x$ and $x'$, and $x$ belongs to $M$. Therefore, in order to conclude that $J$ crosses $M$, it is sufficient to show that $x'$ belongs to $M$.

\medskip \noindent
Let $H$ be a hyperplane separating $x'$ and $y'$. As $x'$ and $y'$ both belong to $N(J)$, necessarily $H$ is transverse to $J$. As a consequence, $H$ cannot separate $x$ and $x'$ nor $y$ and $y'$, since otherwise it would disjoint from $S$ according to Proposition \ref{prop:BigProj}. Therefore, $H$ has to separate $x$ and $y$, and we deduce from Proposition \ref{prop:BigProj} that $H$ separates $Y$ and $Z$. 

\medskip \noindent
We conclude that the hyperplanes separating $x'$ and $y'$ are exactly the hyperplane separating $Y$ and $Z$, hence $d(x',y')=d(Y,Z)$ and finally $x' \in M$ as desired.
\end{proof}

\paragraph{Median triangles.} Unlike CAT(0) cube complexes, a triple of vertices does not necessarily admits a median point in quasi-median graphs. But there is a close notion:

\begin{definition}
Let $X$ be a graph and $x,y,z \in X$ three vertices. A triple of vertices $(x',y',z')$ is a \emph{median triangle} if
$$\left\{ \begin{array}{l} d(x,y)=d(x,x')+d(x',y')+d(y',y) \\ d(x,z)= d(x,x')+d(x',z')+d(z',z) \\ d(y,z)= d(y,y')+d(y',z')+d(z',z) \end{array} \right.,$$
if $d(x',y')=d(y',z')=d(y',z')$, and if $d(x',y')$ is as small as possible. 
\end{definition}

\noindent
As shown in \cite{quasimedian}, quasi-median graphs can be defined in terms of median triangles. In particular, in quasi-median graphs, a triple of vertices always admits a unique median triangle. The following observation, which is a consequence of \cite[Proposition 2.84]{Qm}, will be also needed in the sequel:

\begin{lemma}\label{lem:MedianTriangle}
Let $X$ be a quasi-median graph. Fix three vertices $x,y,z \in X$ and let $(x',y',z')$ denote the median triangle of $(x,y,z)$. The hyperplanes separating two vertices of $\{x',y',z'\}$ are pairwise transverse.
\end{lemma}

\paragraph{Flat rectangles.} Finally, we conclude this subsection by introducing \emph{flat rectangles} in quasi-median graphs, which are useful in the study of hyperbolicity.

\begin{definition}
Let $X$ be a graph. A \emph{flat rectangle} is an isometrically embedded subgraph $R \subset X$ isomorphic to the grid $[0,a] \times [0,b]$ for some $a,b \geq 0$. If $a=b$, $R$ is a \emph{flat square}.
\end{definition}

\noindent
As $R$ is isometrically embedded, for convenience we usually identify $R$ with $[0,a] \times [0,b]$. The relation between flat rectangles and hyperbolicity is made explicit by the following statement:

\begin{lemma}\label{lem:FlatRectangle}\emph{\cite[Proposition 2.113]{Qm}}
A quasi-median graph $X$ is hyperbolic if and only if there exists some $R \geq 0$ such that $X$ does not contain any $R$-thick flat rectangle.
\end{lemma}

\noindent
In practice, flat rectangles are constructed by using the following proposition:

\begin{prop}\label{prop:cycle}
Let $X$ be a quasi-median graph and $(Y_1,Y_2,Y_3,Y_4)$ a cycle of gated subgraphs. There exists a flat rectangle $[0,p] \times [0,q] \subset X$ satisfying $[0,p] \times \{0\} \subset Y_2$, $[0,p] \times \{q\} \subset Y_4$, $\{p\} \times [0,q] \subset Y_3$ and $\{0\} \times [0,q] \subset Y_1$. Moreover, the flat rectangle can be chosen so that any hyperplane intersecting $\{0 \} \times [0,q]$ does not cross $Y_4$. 
\end{prop}

\noindent
This statement is essentially contained in the proof of \cite[Proposition 2.111]{Qm}. We include a sketch of proof for the reader's convenience.

\begin{proof}[Sketch of proof of Proposition \ref{prop:cycle}.]
Let $a \in Y_1 \cap Y_2$ be a vertex minimising the distance to $Y_3 \cap Y_4$. Let $b$ (resp. $d$, $c$) denote its projection onto $Y_3$ (resp. $Y_4$, $Y_3 \cap Y_4$). The proof of \cite[Proposition 2.111]{Qm} constructs a flat rectangle $[0,p] \times [0,q] \subset X$ such that $(0,0)= a$, $(p,0)=b$, $(p,q)=c$ and $d=(0,q)$. Notice that, as a consequence of \cite[Lemma 2.39]{Qm}, $b \in Y_2$ and $d \in Y_1$. Therefore, the convexity of our subgraphs implies that $[0,p] \times \{0\} \subset Y_2$, $[0,p] \times \{q\} \subset Y_4$, $\{p\} \times [0,q] \subset Y_3$ and $\{0\} \times [0,q] \subset Y_1$. Moreover, a hyperplane intersecting $\{0\} \times [0,q]$ separates $a$ from its projection $d$ onto $Y_4$. We deduce from Proposition \ref{prop:BigProj} that such a hyperplane does not cross $Y_4$. 
\end{proof}

\subsection{Graph products and their quasi-median graphs}\label{section:GPandQM}

\noindent
Let $\Gamma$ be a simplicial graph and $\mathcal{G}= \{ G_u \mid u \in V(\Gamma) \}$ be a collection of groups indexed by the vertex-set $V(\Gamma)$ of $\Gamma$. The \emph{graph product} $\Gamma \mathcal{G}$ is defined as the quotient
$$\left( \underset{u \in V(\Gamma)}{\ast} G_u \right) / \langle \langle [g,h]=1, g \in G_u, h \in G_v \ \text{if} \ (u,v) \in E(\Gamma) \rangle \rangle$$
where $E(\Gamma)$ denotes the edge-set of $\Gamma$. The groups of $\mathcal{G}$ are referred to as \emph{vertex-groups}. 

\medskip \noindent
\textbf{Convention.} In all the article, we will assume for convenience that the groups of $\mathcal{G}$ are non-trivial. Notice that it is not a restrictive assumption, since a graph product with some trivial factors can be described as a graph product over a smaller graph all of whose factors are non-trivial.

\medskip \noindent
A \emph{word} in $\Gamma \mathcal{G}$ is a product $g_1 \cdots g_n$ for some $n \geq 0$ and, for every $1 \leq i \leq n$, $g_i \in G$ for some $G \in \mathcal{G}$; the $g_i$'s are the \emph{syllables} of the word, and $n$ is the \emph{length} of the word. Clearly, the following operations on a word does not modify the element of $\Gamma \mathcal{G}$ it represents:
\begin{description}
	\item[Cancellation:] delete the syllable $g_i=1$;
	\item[Amalgamation:] if $g_i,g_{i+1} \in G$ for some $G \in \mathcal{G}$, replace the two syllables $g_i$ and $g_{i+1}$ by the single syllable $g_ig_{i+1} \in G$;
	\item[Shuffling:] if $g_i$ and $g_{i+1}$ belong to two adjacent vertex-groups, switch them.
\end{description}
A word is \emph{reduced} if its length cannot be shortened by applying these elementary moves. Every element of $\Gamma \mathcal{G}$ can be represented by a reduced word, and this word is unique up to the shuffling operation. This allows us to define the \emph{length} of an element $g \in \Gamma \mathcal{G}$, denoted by $|g|$, as the length of any reduced word representing $g$. For more information on reduced words, we refer to \cite{GreenGP} (see also \cite{HsuWise,VanKampenGP}). 

\medskip \noindent
The connection between graph products and quasi-median graphs is made explicit by the following statement \cite[Proposition 8.2, Corollary 8.7]{Qm}:

\begin{thm}
Let $\Gamma$ be a simplicial graph and $\mathcal{G}$ a collection of groups indexed by $V(\Gamma)$. The Cayley graph $X(\Gamma, \mathcal{G}):= \mathrm{Cay} \left( \Gamma \mathcal{G}, \bigcup\limits_{G \in \mathcal{G}} G \backslash \{1 \} \right)$ is a quasi-median graph of cubical dimension $\mathrm{clique}(\Gamma)= \max \{ \# V(\Lambda ) \mid \Lambda \subset \Gamma \ \text{clique} \}$. 
\end{thm}

\noindent
Notice that $\Gamma \mathcal{G}$ naturally acts by isometries on $X(\Gamma, \mathcal{G})$ by left-multiplication and that, as a Cayley graph, the edges of $X(\Gamma, \mathcal{G})$ are naturally labelled by generators, but also by vertices of $\Gamma$ (corresponding to the vertex-group which contains the generator). 

\medskip \noindent
Essentially by construction of the quasi-median graph, we have the following description of its geodesics \cite[Lemma 8.3]{Qm}:

\begin{prop}\label{prop:DistInX}
Let $\Gamma$ be a simplicial graph and $\mathcal{G}$ be a collection of groups indexed by $V(\Gamma)$. Fix two elements $g,h \in \Gamma \mathcal{G}$ and write $g^{-1}h$ as a reduced word $u_1 \cdots u_n$. Then the sequence of vertices $$g,gu_1,gu_1u_2, \ldots, gu_1 \cdots u_n=h$$ defines a geodesic between $g$ and $h$ in $X(\Gamma, \mathcal{G})$. Conversely, any geodesic between $g$ and $h$ is labelled by a reduced word representing $g^{-1}h$.
\end{prop}

\noindent
Notice that, if $\Lambda \subset \Gamma$ is an induced subgraph, then the subgroup $\langle G_u, \ u \in V(\Lambda) \rangle$, which we refer to as a \emph{parabolic subgroup} and which we denote by $\langle \Lambda \rangle$ for short, is naturally isomorphic to the graph product $\Lambda \mathcal{H}$, where $\mathcal{H}= \{ G_u \mid u \in V(\Lambda) \}$. Moreover, thought of as a subgraph of $X(\Gamma, \mathcal{G})$, $\langle \Lambda \rangle$ is naturally isomorphic to $X(\Lambda, \mathcal{H})$.

\paragraph{Hyperplanes of $X(\Gamma, \mathcal{G})$.} For every vertex $u \in V(\Gamma)$, let $J_u$ denote the hyperplane of $X(\Gamma, \mathcal{G})$ containing the clique $\langle u \rangle$. As showed in \cite[Section 8.1]{Qm}, we have the following statement:

\begin{prop}\label{prop:HypInX}
Let $\Gamma$ be a simplicial graph, $\mathcal{G}$ a collection of groups indexed by $V(\Gamma)$ and $J$ a hyperplane of $X(\Gamma, \mathcal{G})$. There exist $g \in \Gamma \mathcal{G}$ and $u \in V(\Gamma)$ such that $J=gJ_u$. Moreover, $N(J)= g \langle \mathrm{star}(u) \rangle$ so that $\mathrm{stab}_{\Gamma \mathcal{G}}( J)= g \langle \mathrm{star}(u) \rangle g^{-1}$.  
\end{prop}

\noindent
Recall that the \emph{star} of a vertex $u \in \Gamma$ is the subgraph of $\Gamma$ generated by $u$ and its neighbors. 

\medskip \noindent
Given a hyperplane $J$ of $X(\Gamma, \mathcal{G})$, one says that $J$ is \emph{labelled} by the vertex $u \in V(\Gamma)$ if $J$ is a translate of $J_u$. The next lemma shows that labels are related to the combinatorics of hyperplanes. 

\begin{lemma}\label{lem:HypTransverseLabel}\emph{\cite[Lemma 8.12]{Qm}}
Let $\Gamma$ be a simplicial graph and $\mathcal{G}$ a collection of groups indexed by $V(\Gamma)$. Two transverse hyperplanes of $X(\Gamma, \mathcal{G})$ must be labelled by adjacent vertices of $\Gamma$, and two tangent hyperplanes must be labelled by two distinct non-adjacent vertices of $\Gamma$. 
\end{lemma}

\noindent
As mentioned above, the graph product $\Gamma \mathcal{G}$ naturally acts by isometries on $X(\Gamma, \mathcal{G})$ by left-multiplication. However, this action is not minimal in general, motivating the following definition and lemma.

\begin{definition}
A quasi-median graph $X$ is \emph{essential} if, for every hyperplane $J$, no sector delimited by $J$ lies in a neighborhood of $N(J)$. 
\end{definition}

\noindent
Our following lemma determines precisely when the quasi-median graph associated to a graph product is essential.

\begin{lemma}\label{lem:Essential}
Let $\Gamma$ be a simplicial graph and $\mathcal{G}$ a collection of groups indexed by $V(\Gamma)$. The quasi-median graph $X(\Gamma, \mathcal{G})$ is essential if and only if $\Gamma$ is not the star of one of its vertices.
\end{lemma}

\begin{proof}
Assume that $u \in V(\Gamma)$ is a vertex which is adjacent to all the other vertices of $\Gamma$. Then $X(\Gamma, \mathcal{G})$ decomposes as the Cartesian product of the clique $\langle u \rangle$ with $\langle \Gamma \backslash \{u \} \rangle$. So $X(\Gamma, \mathcal{G})$ is not essential. 

\medskip \noindent
Conversely, assume that $\Gamma$ is not the star of one of its vertices. It is sufficient to show that, for every $u \in V(\Gamma)$, a sector delimited by $J_u$ does not lie in a neighborhood of $J_u$. So fix a vertex $u \in V(\Gamma)$ and a sector $S$ delimited by $J_u$. Because $G_u$ acts transitively on the collection of the sectors delimited by $J_u$, we may suppose without loss of generality that $S$ contains $1$. By assumption, there exists a vertex $v \in V(\Gamma) \backslash \{u\}$ which is not adjacent to $u$. Fix two non-trivial elements $a \in G_v$ and $b \in G_u$, and an integer $n \geq 1$. According to Proposition \ref{prop:DistInX}, the path 
$$1, \ a, \ ab, \ aba, \ (ab)^2, \ldots, (ab)^n$$
is a geodesic, so it follows from Theorem \ref{thm:GeodesicsQM} that the hyperplanes it crosses, namely 
$$J_v, \ aJ_u, \ abJ_v, abaJ_u, \ (ab)^2J_v, \ldots, (ab)^{n-1}aJ_u,$$
separates $1$ from $(ab)^n$. But we also know from Lemma \ref{lem:HypTransverseLabel} that these hyperplanes cannot be transverse to $J_u$, so they have to separate $(ab)^n$ to $N(J_u)$. Consequently, the distance from $(ab)^n$ to $N(J_u)$ is at least $n$. Because $n$ can be chosen arbitrarily large, the desired conclusion follows.
\end{proof}

\noindent
Let us also record the following observation, which will be used later:

\begin{lemma}\label{lem:CrossSameHyp}
Let $\Gamma$ be a simplicial graph and $\mathcal{G}$ a collection of groups indexed by $V(\Gamma)$. Fix four vertices $a,b,c,d \in X(\Gamma, \mathcal{G})$, and assume that the hyperplanes separating $a$ and $b$ coincide with the hyperplanes separating $c$ and $d$, and that none of them separates $a$ and $c$. Then $d=c \cdot a^{-1}b$. 
\end{lemma}

\begin{proof}
Fix a geodesic $\mu$ from $a$ to $b$, and a geodesic $\nu$ from $c$ to $d$. Say that a pair of edges $(e,f)$ of $\nu$ is \emph{bad} if $\mu$ crosses the hyperplane dual to $f$ and next the hyperplane dual to $e$. 

\medskip \noindent
Fix two edges $e,f \subset \nu$ such that $(e,f)$ is a bad pair and the length of the subsegment of $\nu$ between $e$ and $f$ has minimal length. Notice that, if $e' \subset \nu$ is an edge between $e$ and $f$, then either $(e,e')$ or $(e',f)$ must be a bad pair. So $e$ and $f$ must be adjacent. The fact that $\mu$ crosses the hyperplanes dual to $e$ and $f$ in a different order than $\nu$ (and the fact that no hyperplane separating $a$ and $b$ separates $a$ and $c$) implies that these two hyperplanes are transverse. It follows from Lemma \ref{lem:HypTransverseLabel} that the generators labelling $e$ and $f$ belong to adjacent vertex-groups, so that $e$ and $f$ have to generate a square. Let $\nu'$ denote the path obtained from $\nu$ by replacing $e \cup f$ with the opposite path of length two $e' \cup f'$ in the square generated by $e$ and $f$. By construction, the pair $(e',f')$ is no longer bad in $\nu'$. Moreover, our process did not create additional bad pairs of edges, so that the number of bad pairs of $\nu'$ is smaller that the number of bad pairs of $\nu$. 

\medskip \noindent
As a consequence, by choosing carefully our geodesic $\nu$, we may suppose that $\mu$ and $\nu$ cross their hyperplanes (i.e., the hyperplanes separating $a$ and $b$) in the same order. So the $n$th edge of $\mu$ and the $n$th edge of $\nu$ are dual to the same hyperplane; but they also link the same sectors delimited by this hyperplane, namely the sector containing $a$ to the sector containing $b$, so these two edges have the be labelled by the same generator. Consequently, the paths $\mu$ and $\nu$ are labelled by the same word, say $w$, so that $b=aw$ and $d=cw$ in $\Gamma \mathcal{G}$. 
\end{proof}

\paragraph{Hyperbolicity of $X(\Gamma, \mathcal{G})$.} According to \cite[Fact 8.33]{Qm}, we have the following characterisation:

\begin{prop}\label{prop:hyp}
Let $\Gamma$ be a finite simplicial graph and $\mathcal{G}$ a collection of groups indexed by $V(\Gamma)$. The quasi-median graph $X(\Gamma, \mathcal{G})$ is hyperbolic if and only if $\Gamma$ does not contain any induced square.
\end{prop}

\noindent
In particular, as an immediate consequence of this observation, we deduce the following statement (which also follows from Moussong's characterisation of hyperbolic Coxeter groups):

\begin{cor}\label{cor:hyp}
Let $\Gamma$ be a simplicial graph and $\mathcal{G}$ a collection of finite groups indexed by $V(\Gamma)$. The graph product $\Gamma \mathcal{G}$ is hyperbolic if and only if $\Gamma$ does not contain any induced square.
\end{cor}

\noindent
The proof of Proposition \ref{prop:hyp} is based on Lemma \ref{lem:FlatRectangle} and the following description of the flat rectangles in $X(\Gamma, \mathcal{G})$:

\begin{lemma}\label{lem:FlatRinX}\emph{\cite[Lemma 8.13]{Qm}}
Let $\Gamma$ be a finite simplicial graph and $\mathcal{G}$ a collection of groups indexed by $V(\Gamma)$. An induced subgraph $R \subset X(\Gamma, \mathcal{G})$ is a flat rectangle if and only if there exist a join subgraph $\Lambda_1 \ast \Lambda_2 \leq \Gamma$ and syllables $g_1, \ldots, g_n \in \langle \Lambda_1 \rangle$, $h_1, \ldots, h_m \in \langle \Lambda_2 \rangle$ such that the products $g_1 \cdots g_n$ and $h_1 \cdots h_m$ are reduced and such that $R$ is generated by the vertices 
$$ \{ k g_1 \cdots g_i h_1 \cdots h_j \mid 0 \leq i \leq n, 0 \leq j \leq m\}$$
for some $k \in \Gamma \mathcal{G}$. Moreover, if $R$ is $L$-thick for some $L > \mathrm{clique}(\Gamma)$, then $\Lambda_1$ and $\Lambda_2$ are not complete.
\end{lemma}

\begin{proof}
The first assertion of our lemma is precisely \cite[Lemma 8.13]{Qm}. Next, because $g_1 \cdots g_n$ is a reduced word, it follows that, if $\Lambda_1$ is a complete subgraph, then $g_1, \ldots, g_n$ have to belong to pairwise distinct vertex-groups, hence $n \leq \# V(\Lambda_1) \leq \mathrm{clique}(\Gamma)$. The same holds for $h_1, \ldots, h_m$, proving the second assertion of the lemma. 
\end{proof}

\paragraph{Hausdorff distances in $X(\Gamma, \mathcal{G})$.} Finally, we would to determine when the Hausdorff distance between two cosets of parabolic subgroups is finite. The answer is provided by the following lemma and its corollary.

\begin{lemma}\label{lem:ParabolicHausdorff}
Let $\Gamma$ be a finite simplicial graph, $\mathcal{G}$ a collection of groups indexed by $V(\Gamma)$ and $\Xi \subset \Gamma$ a subgraph. Decompose $\Xi$ as a join $\Xi_0 \ast \Xi_1$ where $\Xi_1$ is a complete and where $\Xi_0$ is not the star of one of its vertices. If $g\langle \Xi \rangle$ lies in the $K$-neighborhood of $\langle \Lambda \rangle$ for some $K \geq 0$ and $g \in \Gamma \mathcal{G}$, then $\Xi_0 \subset \Lambda$ and there exists some $h \in \langle \Lambda \rangle$ such that the Hausdorff distance between $g \langle \Xi \rangle$ and $h \langle \Xi \rangle$ is at most $K+ 2 \cdot \mathrm{clique}(\Gamma)$. 
\end{lemma}

\begin{proof}
Fix a vertex $u \in V(\Xi_0)$. So the hyperplane $gJ_u$ crosses $g \langle \Xi_0 \rangle$. If it does not cross $\langle \Lambda \rangle$, then there exists a sector $S$ delimited by $gJ_u$ which is disjoint from $\langle \Lambda \rangle$. But we know from Lemma \ref{lem:Essential} that $g \langle \Xi_0 \rangle$ is essential, so it contains vertices in $S$ arbitrarily far away from $gJ_u$, and so arbitrarily far away from $\langle \Lambda \rangle$, contradicting the fact that $g \langle \Xi_0 \rangle$ lies in a neighborhood of $\langle \Lambda \rangle$. Thus, we have proved:

\begin{fact}\label{fact:CrossCross}
Any hyperplane crossing $g \langle \Xi_0 \rangle$ crosses $\langle \Lambda \rangle$.
\end{fact}

\noindent
Because the hyperplanes crossing $g \langle \Xi_0 \rangle$ are labelled by vertices of $\Xi_0$, and those crossing $\langle \Lambda \rangle$ are labelled by vertices of $\Lambda$, we conclude that $\Xi_0 \subset \Lambda$.

\medskip \noindent
Next, let $Y \subset g \langle \Xi_0 \rangle$ denote the subgraph generated by the vertices $y \in g \langle \Xi_0 \rangle$ satisfying $d(y,\langle \Lambda \rangle)= d(g \langle \Xi_0 \rangle, \langle \Lambda \rangle)$. If $g \langle \Xi_0 \rangle$ contains $Y$ properly, then, by considering a hyperplane separating a vertex of $g \langle \Xi_0 \rangle \backslash Y$ from $Y$ (which exists according to Proposition~\ref{prop:BigProj} and Lemma \ref{lem:GatedBridge}), we would deduce from Lemma \ref{lem:GatedBridge} that there exists a hyperplane crossing $g \langle \Xi_0 \rangle$ which does not cross $\langle \Lambda \rangle$, contradicting Fact \ref{fact:CrossCross}. Therefore, $Y= g \langle \Xi_0 \rangle$. 

\medskip \noindent
For convenience, let $p : X \to \langle \Lambda \rangle$ denote the projection onto $\langle \Lambda \rangle$. Fix a vertex $x \in g \langle \Xi_0 \rangle$, and write $x=g \xi$ where $\xi \in \langle \Xi_0 \rangle$. It follows from Proposition \ref{prop:BigProj} that the hyperplanes separating $x$ and $p(x)$ (or $g$ and $p(g)$) separates $g \langle \Xi_0 \rangle$ and $\langle \Lambda \rangle$. Consequently, $x$ and $g$ are separated by the same hyperplanes as $p(x)$ and $p(g)$ (and none of them separates $g$ and $p(g)$). It follows from Lemma \ref{lem:CrossSameHyp} that $p(x)=p(g) \xi$. Thus, the map
$$\left\{ \begin{array}{ccc} g \langle \Xi_0 \rangle & \to & \langle \Lambda \rangle \\ g \xi & \mapsto & p(g) \xi \end{array} \right.$$
sends a vertex of $g \langle \Xi_0 \rangle$ to a vertex of $\langle \Lambda \rangle$ within distance $d(g \langle \Xi_0 \rangle, \langle \Lambda \rangle) \leq K$. Moreover, the image of this map is clearly $p(g) \langle \Xi_0 \rangle$, so the Hausdorff distance between $g \langle \Xi_0 \rangle$ and $p(g) \langle \Xi_0 \rangle$ is at most $K$, where $p(g)$ belongs to $\langle \Lambda \rangle$. 

\medskip \noindent
As the Hausdorff distance between $\langle \Xi \rangle$ and $\langle \Xi_0 \rangle$ is $\#V(\Xi_1) \leq \mathrm{clique}(\Gamma)$, since $\langle \Xi \rangle$ decomposes as the Cartesian product of $\langle \Xi_0 \rangle$ with the prism $\langle \Xi_1 \rangle$, we conclude that the Hausdorff distance between $g \langle \Xi \rangle$ and $p(g) \langle \Xi \rangle$ is at most $K+2 \cdot \mathrm{clique}(\Gamma)$.
\end{proof}

\begin{cor}\label{cor:ParabolicHausdorff}
Let $\Gamma$ be a finite simplicial graph, $\mathcal{G}$ a collection of groups indexed by $V(\Gamma)$ and $\Phi, \Psi \leq \Gamma$ two induced subgraphs. Decompose $\Phi$ (resp. $\Psi$) as a join $\Phi_0 \ast \Phi_1$ (resp. $\Psi_0 \ast \Psi_1$) where $\Phi_1$ (resp. $\Psi_1$) is complete and where $\Phi_0$ (resp. $\Psi_0$) is not the star of one its vertices. The Hausdorff distance between $\langle \Phi \rangle$ and $\langle \Psi \rangle$ is finite if and only if $\Phi_0= \Psi_0$. 
\end{cor}

\begin{proof}
The Hausdorff distance between $\langle \Phi \rangle$ and $\langle \Phi_0 \rangle$ is finite as $\langle \Phi \rangle$ decomposes as the Cartesian product of $\langle \Phi_1 \rangle$ with the prism $\langle \Phi_0 \rangle$. Similarly, the Hausdorff distance between $\langle \Psi \rangle$ and $\langle \Psi_0 \rangle$ is finite. By applying Lemma \ref{lem:ParabolicHausdorff} twice, we know that if the Hausdorff distance between $\langle \Psi_0 \rangle$ and $\langle \Phi_0 \rangle$ is finite then $\Psi_0 \subset \Phi_0$ and $\Phi_0 \subset \Psi_0$. We conclude that, if the Hausdorff distance between $\langle \Phi \rangle$ and $\langle \Psi \rangle$ is finite, then $\Phi_0 = \Psi_0$. The converse is clear.
\end{proof}

\subsection{Gated Morse subgraphs}\label{section:gated}

\noindent
In this subsection, we are interested in Morse subspaces in quasi-median graphs. The first observation is that, up to finite Hausdorff distance, it may always be assumed that the subspace we are looking at is a gated subgraph. More precisely:

\begin{lemma}\label{lem:DistN}
Let $X$ be a quasi-median graph whose cubical dimension is finite and $Y$ a subspace. If $Y$ is Morse then the Hausdorff distance between $Y$ and its gated hull is finite and depends only on $\dim_\square(X)$ and the Morse-gauge of $Y$.
\end{lemma}

\noindent
The proof is an immediate consequence of the following lemma (proved in \cite[Theorem H]{MR2413337} for uniformly locally finite CAT(0) cube complexes and in \cite[Lemma 4.3]{MoiHypCube} in full generality), where $\mathrm{Ram}(\cdot)$ denotes the \emph{Ramsey number}. Recall that, if $n \geq 0$, $\mathrm{Ram}(n)$ is the smallest integer $k \geq 0$ satisfying the following property: if one colors the edges of a complete graph containing at least $k$ vertices with two colors, it is possible to find a monochromatic complete subgraph containing at least $n$ vertices. Often, it is used to find a subcollection of pairwise disjoint hyperplanes in a collection of hyperplanes of some finite-dimensional CAT(0) cube complex (see for instance \cite[Lemma 3.7]{coningoff}); the same can be done for hyperplanes in quasi-median graphs of finite cubical dimension. 

\begin{lemma}\label{lem:convexhull}
Let $X$ be a quasi-median graph of finite cubical dimension and $S \subset X$ a set of vertices which is $K$-quasiconvex. Then the gated hull of $S$ is included into the $\mathrm{Ram}( \max(\dim(X),K)+1)$-neighborhood of $S$.
\end{lemma}

\begin{proof}
Let $x \in X$ be a vertex which belongs to the gated hull of $S$, and let $p \in S$ be a vertex of $S$ which minimises the distance to $x$. If $d(p,x) \geq \mathrm{Ram}(n)$ for some $n \geq \dim_\square(X)+1$, then there exists a collection of hyperplanes $J_1, \ldots, J_n$ separating $x$ and $p$ such that, for every $2 \leq i \leq n-1$, $J_i$ separates $J_{i-1}$ and $J_{i+1}$. Because $x$ belongs to the gated hull of $S$, no hyperplane separates $x$ from $S$. Therefore, there exists some $y \in S$ such that $J_1, \ldots, J_n$ separate $p$ and $y$. Let $(x',y',p')$ denote the median triangle of $(x,y,p)$. Because $p'$ belongs to a geodesic between $x$ and $p$ and that $d(x,p)=d(x,S)$, necessarily $d(p',p)=d(p',S)$. On the other hand, $p'$ also belongs to a geodesic between $y,p \in S$, so the $K$-quasiconvexity of $S$ implies $d(p',S) \leq K$, hence $d(p',p) \leq K$. 

\medskip \noindent
Notice that, because a geodesic crosses a hyperplane at most once according to Theorem~\ref{thm:GeodesicsQM}, the hyperplanes $J_1, \ldots, J_n$  either separate $p$ and $p'$ or separate $\{x',y',p'\}$. But we know from Lemma \ref{lem:MedianTriangle} that the hyperplanes separating $\{x',y',p'\}$ are pairwise transverse, so at least $n-1$ hyperplanes among $J_1, \ldots, J_n$ has to separate $p$ and $p'$. We conclude that $n \leq K+1$.
\end{proof}

\noindent
The main statement of this subsection is the following criterion, which we proved in \cite{article3} for CAT(0) cube complexes.

\begin{prop}\label{prop:contracting}
Let $X$ be a quasi-median graph of finite cubical dimension and $Y$ a gated subgraph. The following assertions are equivalent:
\begin{itemize}
	\item[(i)] $Y$ is a contracting subgraph;
	\item[(ii)] $Y$ is a Morse subgraph;
	\item[(iii)] there exists a constant $C \geq 0$ such that, for every flat square $[0,r] \times [0,r] \subset X$ satisfying $[0,r] \times \{0 \} \subset Y$, the side $[0,r] \times \{r\}$ lies in the $C$-neighborhood of $Y$;
	\item[(iv)] there exists a constant $C \geq 0$ such that, for every grid of hyperplanes $(\mathcal{H}, \mathcal{V})$ satisfying $\mathcal{V} \subset \mathcal{H}(Y)$ and $\mathcal{H} \cap \mathcal{H}(Y)= \emptyset$, one has $\min(\# \mathcal{H}, \# \mathcal{V}) < C$. 
\end{itemize}
\end{prop}

\noindent
Recall that, given a metric space $S$, a subspace $R \subset S$ is \emph{contracting} if there exists some $D \geq 0$ such that the nearest-point projection of any ball disjoint from $R$ onto $R$ has diameter at most $D$. A slight variation of this definition is:

\begin{lemma}\label{lem:ContractingMinus}\emph{\cite[Lemma 2.18]{article3}}
Let $X$ be a geodesic metric space, $S \subset X$ a subspace and $L \geq 0$ a constant. Then $S$ is contracting if and only if there exists $C \geq 0$ such that, for all point $x, y \in X$ satisfying $d(x, y) < d(x, S) - L$, the nearest-point projection of $\{x, y \}$ onto $S$ has diameter at most $C$.
\end{lemma}

\noindent
We are now ready to prove our proposition.

\begin{proof}[Proof of Proposition \ref{prop:contracting}.]
It is proved in \cite[Lemma 3.3]{Sultan} that, in any geodesic metric spaces, a contracting quasi-geodesic always defines a Morse subspace. In fact, the proof does not depend on the fact that the contracting subspace we are looking at is a quasi-geodesic, so that being a contracting subspace always implies being a Morse subspace. In particular, the implication $(i) \Rightarrow (ii)$ holds.

\medskip \noindent
Assume that $Y$ is a Morse subgraph and let $C$ be such that any $(3,0)$-quasigeodesic between two points of $Y$ stays in the $C$-neighborhood of $Y$. Now let $[0,r] \times [0,r] \subset X$ be a flat square satisfying $[0,r] \times \{ 0 \} \subset Y$. Because flat squares are isometrically embedded by definition, it follows that the path
$$\left( \{0\} \times [0,r] \right) \cup \left( [0,r] \times \{r\} \right) \cup \left( \{r\} \times [0,r] \right)$$
defines a $(3,0)$-quasigeodesic between the two vertices $(0,0)$ and $(r,r)$ of $Y$. We conclude that $[0,r] \times \{r\}$ lies in the $C$-neighborhood of $Y$, showing the implication $(ii) \Rightarrow (iii)$. 

\medskip \noindent
Assume that $(iii)$ holds and let $C$ be the corresponding constant. Now, let $(\mathcal{H}, \mathcal{V})$ be a grid of hyperplanes such that $\mathcal{V} \subset \mathcal{H}(Y)$ and $\mathcal{H} \cap \mathcal{H}(Y)= \emptyset$. For convenience, write $\mathcal{V}=\{ V_1, \ldots, V_n \}$ such that $V_i$ separates $V_{i-1}$ and $V_{i+1}$ for every $2 \leq i \leq n-1$; and $\mathcal{H}= \{ H_1, \ldots, H_m \}$ such that $H_i$ separates $H_{i-1}$ and $H_{i+1}$ for every $2 \leq i \leq m-1$, and such that $H_1$ separates $Y$ and $H_m$. Consider the cycle of subgraphs $(N(V_1), Y, N(V_n), N(H_m))$. According to Proposition \ref{prop:cycle}, there exists a flat rectangle $[0,p] \times [0,q] \subset X$ such that $[0,p] \times \{ 0 \} \subset Y$, $\{0 \} \times [0,q] \subset N(V_1)$, $\{p\} \times [0,q] \subset N(V_n)$ and $[0,p] \times \{q \} \subset N(H_m)$, and such that the hyperplanes intersecting $\{0\} \times [0,q]$ do not cross $Y$. 

\medskip \noindent
First, notice that, as the hyperplanes $H_1, \ldots, H_{m-1}$ do not cross $Y$, they must separate $\{0\} \times [0,q]$ from $Y$, hence $q \geq m-1$. Also, as $V_2, \ldots, V_{n-1}$ separate $V_1$ and $V_n$, necessarily $p \geq n-2$. Thus, we have shown the following statement, which we record for future use:

\begin{fact}\label{fact:FlatRectangle}
Let $X$ be a quasi-median graph, $Y \subset X$ a gated subgraph, and $(\mathcal{H}, \mathcal{V})$ a grid of hyperplanes satisfying $\mathcal{V} \subset \mathcal{H}(Y)$ and $\mathcal{H} \cap \mathcal{H}(Y)= \emptyset$. Then there exist two integers $p \geq \# \mathcal{V}-2$, $q \geq \# \mathcal{H} -1$ and a flat rectangle $[0,p] \times [0,q] \subset X$ such that $[0,p] \times \{ 0 \} \subset Y$ and such that the hyperplanes intersecting $\{0\} \times [0,q]$ do not cross $Y$. 
\end{fact} 

\noindent
Next, because the hyperplanes intersecting $\{0\} \times [0,q]$ do not cross $Y$, the distance from a vertex of $[0,p] \times \{k\}$ to $Y$ must be $k$ for every $0 \leq k \leq q$. Therefore, by applying $(iii)$ to the flat square $[0,\min(p,q)] \times [0,\min(p,q)]$, we deduce that $\min(p,q) \leq C$. 

\medskip \noindent
Consequently, we have $\min(n,m) \leq \min(p,q) +2 \leq C+2$, concluding the proof of the implication $(iii) \Rightarrow (iv)$. 

\medskip \noindent
Finally, let us turn to the proof of $(iv) \Rightarrow (i)$. So we suppose that there exists a constant $C \geq 0$ such that, for every grid of hyperplanes $(\mathcal{H}, \mathcal{V})$ satisfying $\mathcal{V} \subset \mathcal{H}(Y)$ and $\mathcal{H} \cap \mathcal{H}(Y)= \emptyset$, one has $\min(\# \mathcal{H}, \# \mathcal{V}) < C$. And we fix two vertices $x,y \in X$ satisfying $d(x,y) < d(x,Y)-L$ where $L= \mathrm{Ram}( \max(C, \dim_\square(X))+1)$. We claim that the distance between the projections $x'$ and $y'$ respectively of $x$ and $y$ onto $Y$ are within distance $L$. 

\medskip \noindent
Because there exist $d(x,Y)$ hyperplanes separating $x$ from $Y$, and that there exist less than $d(x,Y)-L$ hyperplanes separating $x$ and $y$, necessarily there must exist at least $L$ hyperplanes separating $\{x,y\}$ from $Y$. Consequently, there exists a collection $\mathcal{H}$ of at least $C+1$ pairwise non-transverse hyperplanes separating $\{x,y\}$ and $Y$. Let $\mathcal{V}_0$ denote the collection of the hyperplanes separating $x'$ and $y'$. If $\# \mathcal{V}_0 \leq \mathrm{Ram}(\dim_\square(X)+1)$ then we deduce from Theorem \ref{thm:GeodesicsQM} that
$$d(x',y')= \# \mathcal{V}_0 \leq \mathrm{Ram}(\dim_\square(X)+1) \leq L$$
and we are done. Otherwise, let $k > \dim_\square(X)$ be the smallest integer such that $\# \mathcal{V}_0 \geq \mathrm{Ram}(k)$. So $\mathcal{V}_0$ contains a subcollection $\mathcal{V}$ of $k$ pairwise non-transverse hyperplanes. Notice that, as a consequence of Proposition \ref{prop:BigProj}, the hyperplanes of $\mathcal{V}$ separates $\{x,x'\}$ and $\{y,y'\}$, which implies that any hyperplane of $\mathcal{H}$ must be transverse to any hyperplane of $\mathcal{V}$. In other words, $(\mathcal{H}, \mathcal{V})$ is a grid of hyperplanes. By applying $(iv)$, we find that $\min(\# \mathcal{H}, \# \mathcal{V}) \leq C$. But we already know that $\# \mathcal{H} \geq C+1$, so $k= \# \mathcal{V} \leq C$. Consequently,
$$d(x',y')= \# \mathcal{V}_0 \leq \mathrm{Ram}(k+1) \leq \mathrm{Ram}(C+1) \leq L.$$
We conclude thanks to Lemma \ref{lem:ContractingMinus} that $Y$ is contracting.
\end{proof}

\noindent
As an application of Proposition \ref{prop:contracting}, we are able to determine precisely when a parabolic subgroup of a graph product defines a Morse subgraph in the corresponding quasi-median graph. More precisely:

\begin{prop}\label{prop:ParabolicMorse}
Let $\Gamma$ be a simplicial graph, $\Lambda \subset \Gamma$ an induced subgraph and $\mathcal{G}$ a collection of groups indexed by $V(\Gamma)$. Then $\langle \Lambda \rangle \subset X(\Gamma, \mathcal{G})$ is a Morse subgraph if and only if $\Lambda$ is square-complete.
\end{prop}

\begin{proof}
Suppose that $\langle \Lambda \rangle$ is not a Morse subgraph. As a consequence of Proposition \ref{prop:contracting} and Fact \ref{fact:FlatRectangle}, there exist two integers $p,q > \mathrm{clique}(\Gamma)$ and a flat rectangle $[0,p] \times [0,q] \subset X(\Gamma, \mathcal{G})$ such that $[0,p] \times \{ 0 \} \subset \langle \Lambda \rangle$ and such that a hyperplane intersecting $\{0\} \times [0,q]$ does not cross $\langle \Lambda \rangle$. Up to translating by an element of the subgroup $\langle \Lambda \rangle$, we may suppose without loss of generality that $(0,0)=1$. 

\medskip \noindent
Let $a_1 \cdots a_q$ denote the reduced word labelling the geodesic $\{0\} \times [0,q]$ (from $(0,0)$ to $(0,q)$), where $a_1 \in G_{u_1}, \ldots, a_q \in G_{u_q}$ are elements and $u_1, \ldots, u_q \in V(\Gamma)$ vertices. Because $q > \mathrm{clique}(\Gamma) $, there must exist $1 \leq i <j \leq q$ such that $u_i$ and $u_j$ are not adjacent in $\Gamma$. Without loss of generality, we may suppose that $u_i$ is adjacent to $u_k$ for every $1 \leq k <i$. It follows from Proposition \ref{prop:HypInX} that $a_1 \cdots a_{i-1}J_{u_i}=J_{u_i}$. Since $1=(0,0) \in \langle \Lambda \rangle$ but $J_{u_i} \notin  \mathcal{H}(\langle \Lambda \rangle)$, we must have $u_i \notin \Lambda$. 

\medskip \noindent
Similarly, because $p > \mathrm{clique}(\Gamma)$, there must exist two edges of $[0,p] \times \{ 0 \} \subset \langle \Lambda \rangle$ labelled by non-adjacent vertices of $\Lambda$, say $u$ and $v$. By noticing that any hyperplane intersecting $[0,p] \times \{0\}$ must be transverse to any hyperplane intersecting $\{0\} \times [0,q]$, it follows that $u$ and $v$ are adjacent to both $a_i$ and $a_j$. In other words, $a_i,a_j,u,v$ define an induced square of $\Gamma$ such that $u,v \in \Lambda$ are diametrically opposite but $a_i \notin \Lambda$. Thus, we have proved that $\Lambda$ is not square-complete.

\medskip \noindent
Conversely, suppose that there exists some induced square in $\Gamma$ with two diametrically opposite vertices $u$ and $v$ in $\Lambda$ but with one of its two other vertices, say $a$, not in $\Lambda$. Let $b$ denote the fourth vertex of our square and fix four non-trivial elements $\alpha \in G_a$, $\beta \in G_b$, $\mu \in G_u$, $\nu \in G_v$. Consider the two infinite geodesic rays
$$1, \ \mu, \ \mu \nu, \ (\mu \nu) \mu, \ (\mu \nu)^2, \ldots, \ (\mu \nu)^n, \ldots$$
and
$$1, \ \alpha, \ \alpha \beta, \ (\alpha \beta) \alpha, \ (\alpha \beta)^2, \ldots, \ (\alpha \beta)^n, \ldots$$
say $r_1$ and $r_2$ respectively. Since $u$ and $v$ commute with both $a$ and $b$, it follows that $r_1$ and $r_2$ bound a copy of $[0,+ \infty) \times [0,+ \infty)$ (which is generated by the vertices $\zeta \xi$ where $\zeta$ and $\xi$ are prefixes of the infinite words $(\mu \nu)^\infty$ and $(\alpha \beta)^{\infty}$ respectively). As a consequence, for every $n \geq 1$, any hyperplane of $\mathcal{H}_n= \{ (\alpha \beta)^kJ_a \mid k \leq n \}$ is transverse to any hyperplane of $\mathcal{V}_n = \{ (\mu \nu)^k J_u \mid k \leq n\}$. Moreover, notice that $\mathcal{H}_n$ and $\mathcal{V}_n$ do not contain facing triples since they are collections of hyperplanes transverse to the geodesic rays $r_2$ and $r_1$ respectively; and $\mathcal{H}_n \cap \mathcal{H}(\langle \Lambda \rangle) = \emptyset$ since $a \notin \Lambda$; and of course $\mathcal{V}_n \subset \mathcal{H}(\langle \Lambda \rangle)$. It follows from Proposition \ref{prop:contracting} that $\langle \Lambda \rangle$ is not a Morse subgraph. 
\end{proof}

\noindent
As a particular case, one gets the following statement, which generalises the case of right-angled Coxeter groups proved in \cite[Proposition 4.9]{MoiHypCube}. 

\begin{cor}
Let $\Gamma$ be a simplicial graph, $\Lambda \subset \Gamma$ an induced subgraph and $\mathcal{G}$ a collection of finite groups indexed by $V(\Gamma)$. Then $\langle \Lambda \rangle$ is a Morse subgroup if and only if $\Lambda$ is square-complete.
\end{cor}

\begin{proof}
The corollary is a direct consequence of Proposition \ref{prop:ParabolicMorse} and of the fact that, because vertex-groups are finite, $X(\Gamma, \mathcal{G})$ is a Cayley graph of $\Gamma \mathcal{G}$ constructed from a finite generating set. 
\end{proof}

\noindent
As an other application of Proposition \ref{prop:contracting}, we prove the following observation, which will be fundamental in the next section:

\begin{lemma}\label{lem:InterInclusion}
Let $X$ be a quasi-median graph of finite cubical dimension and $Y$ a Morse gated subgraph. There exists a constant $C \geq 1$ such that, for every gated subgraph $P$ which decomposes as a Cartesian product of two essential unbounded quasi-median graphs, if $| \mathcal{H}(Y) \cap \mathcal{H}(P)| \geq C$ and $Y \cap P \neq \emptyset$ then $P \subset Y$.
\end{lemma}

\begin{proof}
Let $C$ denote the constant given by the point (iv) of Proposition \ref{prop:contracting}, and let $P \subset X$ be a gated subgraph which decomposes as a Cartesian product $P_1 \times P_2$ of two essential unbounded quasi-median graphs $P_1,P_2$. Notice that $\mathcal{H}(P)= \mathcal{H}(P_1) \sqcup \mathcal{H}(P_2)$. We assume that $|\mathcal{H}(Y) \cap \mathcal{H}(P)| \geq 2 \cdot \mathrm{Ram}(\max(\dim_\square(X),C)+1)$. There exists $i \in \{ 1,2 \}$ such that 
$$|\mathcal{H}(Y) \cap \mathcal{H}(P_i) | \geq \frac{1}{2} | \mathcal{H}(Y) \cap \mathcal{H}(P)| \geq \mathrm{Ram}(\max(\dim_\square(X),C)+1).$$
As a consequence, there exist hyperplanes $V_0, \ldots, V_C \in \mathcal{H}(Y) \cap \mathcal{H}(P_i)$ such that $V_k$ separates $V_{k-1}$ and $V_{k+1}$ for every $1 \leq k \leq C-1$. Without loss of generality, suppose that $i=1$. Let $J \in \mathcal{H}(P_2)$ be a hyperplane. If $J \notin \mathcal{H}(Y)$, let $J^+$ denote a sector delimited by $J$ which is disjoint from $Y$. Since $P_2$ essential and unbounded, there exist hyperplanes $J_1, J_2, \ldots \in \mathcal{H}(P_2)$ such that $J_k$ separates $J_{k-1}$ and $J_{k+1}$ for every $k \geq 2$ and such that $J_k$ separates $J$ and $J_{k+1}$ for every $k \geq 1$. Notice that, for every $k \geq 1$, the hyperplane $J_k$ is disjoint from $Y$; on the other hand, it is transverse to any hyperplane of $\mathcal{H}(P_1)$ and a fortiori to $V_0, \ldots, V_C$. Consequently, $\mathcal{V}= \{V_0, \ldots, V_C\}$ and $\mathcal{H}=\{J, J_1, \ldots, J_{C} \}$ define a grid of hyperplanes satisfying $\mathcal{V} \subset \mathcal{H}(Y)$, $\mathcal{H} \cap \mathcal{H}(Y) = \emptyset$ and $\min(\# \mathcal{H}, \# \mathcal{V}) \geq C+1$. This contradicts the definition of $C$. Therefore, $J$ necessarily belongs to $\mathcal{H}(Y)$. 

\medskip \noindent
Thus, we have proved that $\mathcal{H}(P_2) \subset \mathcal{H}(Y)$. By switching $P_1$ and $P_2$ in the previous argument, one shows similarly that $\mathcal{H}(P_1) \subset \mathcal{H}(Y)$. Therefore, $\mathcal{H}(P) \subset \mathcal{H}(Y)$ and $Y \cap P \neq \emptyset$. The desired conclusion follows from the following observation:

\begin{fact}
Let $X$ be a quasi-median graph and $P,Y \subset X$ two gated subgraphs. Suppose that $\mathcal{H}(P) \subset \mathcal{H}(Y)$ and that $Y \cap P \neq \emptyset$. Then $P \subset Y$. 
\end{fact}

\noindent
Suppose that $P \nsubseteq Y$, i.e., there exists a vertex $x \in P$ which does not belong to $Y$. Let $J$ be a hyperplane separating $x$ from its projection onto $Y$. According to Proposition~\ref{prop:BigProj}, $J$ separates $x$ from $Y$. Two cases may happen. Either $J$ belongs to $\mathcal{H}(P)$, so that $\mathcal{H}(P) \nsubseteq \mathcal{H}(Y)$; or $J$ does not belong to $\mathcal{H}(P)$, so that $J$ must separate $P$ and $Y$, which implies that $P \cap Y= \emptyset$. 
\end{proof}

\begin{remark}\label{remark:MorseRAAG}
Let $G$ be a group acting properly and cocompactly on a quasi-median graph $X$. Assume that the neighborhood of any hyperplane decomposes as the Cartesian product of two unbounded essential subgraphs, and that the \emph{crossing graph} of $X$ (i.e., the graph whose vertices are the hyperplanes of $X$ and whose edges link two hyperplanes if they are transverse) is connected. Then it follows from Lemma \ref{lem:InterInclusion} that infinite-index Morse subgroups of $G$ are virtually free. 

\medskip \noindent
Indeed, let $H \leq G$ be a Morse subgroup. As a consequence of Lemma \ref{lem:DistN}, there is a gated subgraph $Y \subset X$ on which $H$ acts cocompactly. If all the hyperplanes of $Y$ are bounded (in fact, uniformly bounded as $H$ acts on $Y$ with finitely many orbits of hyperplanes), then it is not difficult to show that $Y$ must be a quasi-tree (for instance, reproduce word for word the proof of \cite[Proposition 3.8]{coningoff} written for CAT(0) cube complexes), and we conclude that $H$ must be virtually free (see \cite[Th\'eor\`eme~7.19]{GhysHyp}). Next, if $Y$ contains an unbounded hyperplane $J$, it follows from Lemma~\ref{lem:InterInclusion} that $N(J) \subset Y$. Similarly, if $J'$ is transverse to $J$, then it follows from Lemma \ref{lem:InterInclusion} again that $N(J') \subset Y$. And so on. Because the crossing graph is connected, we conclude that $Y$ contains the neighborhood of every hyperplane of $X$, hence $X=Y$. Therefore, $H$ must be a finite-index subgroup. 

\medskip \noindent
For instance, this criterion applies to freely irreducible right-angled Artin groups, providing a simple proof of the fact that infinite-index Morse subgroups in these groups must be free. Alternative arguments can be found in \cite{StronglyQC,MoiHypCube}.
\end{remark}

\section{Characterisation of eccentric subspaces}\label{section:eccentric}

\noindent
Recall from the introduction that:

\begin{definition}
Let $X$ be a geodesic metric space. A subspace $Y \subset X$ is \emph{eccentric} if it is Morse, non-hyperbolic, and if, for every map $m : (0,+ \infty) \times [0,+ \infty) \to \mathbb{R}$, the Hausdorff distance between $Y$ and any non-hyperbolic Morse subspace $Z \subset Y$ with Morse-gauge $m$ is bounded above by a finite constant $E(m)$. The map $E$ is referred to as the \emph{eccentric-gauge} of $Y$. 
\end{definition}

\noindent
In this section, our goal is to prove the main result of the article, namely we want to characterise eccentric subspaces in graph products of finite groups. Before stating this characterisation, we need some vocabulary.

\begin{definition}
Let $\Gamma$ be a simplicial graph. A subgraph $\Lambda$ is \emph{square-complete} if it is induced and if every induced square intersecting $\Lambda$ along at least two opposite vertices is included into $\Lambda$. A \emph{minsquare subgraph} is a minimal square-complete subgraph which contains at least one induced square. 
\end{definition}

\noindent
Our main theorem is: 

\begin{thm}\label{thm:eccentricsub}
Let $\Gamma$ be a finite simplicial graph and $\mathcal{G}$ a collection of finite groups indexed by $V(\Gamma)$. A subspace $M \subset \Gamma \mathcal{G}$ is eccentric if and only if there exists a minsquare subgraph $\Lambda \subset \Gamma$ such that $M$ is at finite Hausdorff distance from a coset of $\langle \Lambda \rangle$.
\end{thm}

\noindent
We begin by proving three preliminary lemmas. The first one is particular case of the theorem for parabolic subgroups.

\begin{lemma}\label{lem:ParabolicEccentric}
Let $\Gamma$ be a finite simplicial graph, $\mathcal{G}$ a collection of groups indexed by $V(\Gamma)$ and $\Lambda \leq \Gamma$ a subgraph. If $\langle \Lambda \rangle$ is an eccentric subspace then $\Lambda$ decomposes as the join of a minsquare subgraph $\Lambda_0$ and a complete subgraph. Moreover, the Hausdorff distance between $\langle \Lambda \rangle$ and $\langle \Lambda_0 \rangle$ is at most $\mathrm{clique}(\Gamma)$.
\end{lemma}

\begin{proof}
Decompose $\Lambda$ as a join $\Lambda = \Lambda_0 \ast \Lambda_1$ where $\Lambda_1$ is complete and where $\Lambda_0$ is not the star of one of its vertices. Notice that, as $\langle \Lambda \rangle$ decomposes as the Cartesian product of $\langle \Lambda_0 \rangle$ with the prism $\langle \Lambda_1 \rangle$, the Hausdorff distance between $\langle \Lambda \rangle$ and $\langle \Lambda_0 \rangle$ is at most $\mathrm{clique}(\Gamma)$. Consequently, $\langle \Lambda_0 \rangle$ must be an eccentric subspace as well. We claim that $\Lambda_0$ is a minsquare subgraph of $\Gamma$. 

\medskip \noindent
So let $\Xi \subset \Lambda_0$ be a square-complete subgraph which contains an induced square. It follows from Propositions \ref{prop:hyp} and \ref{prop:ParabolicMorse} that $\langle \Xi \rangle$ is non-hyperbolic Morse subgraph of $X(\Gamma, \mathcal{G})$, so the Hausdorff distance between $\langle \Xi \rangle$ and $\langle \Lambda_0 \rangle$ must be finite. It follows from Corollary \ref{cor:ParabolicHausdorff} that $\Lambda_0 \subset \Xi$, hence $\Lambda_0= \Xi$. Thus, we have proved that $\Lambda_0$ is a minsquare subgraph. 
\end{proof}

\noindent
Loosely speaking, our second preliminary lemma shows that a Morse subspace in a graph product of finite groups must contain a ``nice flat''.

\begin{lemma}\label{lem:SquareEccentric}
Let $\Gamma$ be a finite simplicial graph, $\mathcal{G}$ a collection of groups indexed by $V(\Gamma)$, and $M \subset X(\Gamma, \mathcal{G})$ a non-hyperbolic Morse subgraph. There exists an induced square $\Lambda \subset \Gamma$ and an element $g \in \Gamma \mathcal{G}$ such that the subgraph $g \langle \Lambda \rangle$ lies in the $K$-neighborhood of $M$, where $K$ is a constant depending only on $\Gamma$ and the Morse-gauge of $M$.
\end{lemma}

\begin{proof} 
Let $M^+$ denote the gated hull of the $\mathrm{clique}(\Gamma)$-neighborhood of $M$; notice that, as a consequence Lemma~\ref{lem:DistN}, the Hausdorff distance between $M$ and $M^+$ is finite and only depends on $\Gamma$ and the Morse-gauge of $M$. Our goal is to construct an induced square $\Lambda \subset \Gamma$ such that $\langle \Lambda \rangle \subset M^+$.

\medskip \noindent
Since $M$ is not hyperbolic, we know from Lemma \ref{lem:FlatRectangle} that $M$ contains an $L$-thick flat rectangle $R$ where $L=C+ \mathrm{clique}(\Gamma)$ with $C$ the constant given by Lemma~\ref{lem:InterInclusion} for $M^+$. According to Lemma \ref{lem:FlatRinX}, there exist two induced subgraphs $\Lambda_1, \Lambda_2 \subset \Gamma$ and an element $g \in \Gamma \mathcal{G}$ such that $R \subset g \langle \Lambda_1 \ast \Lambda_2 \rangle$; up to translating by $g^{-1}$, we will suppose without loss of generality that $g=1$. The subgraph $\Lambda_1 \ast \Lambda_2$ can be written as $\Lambda_1' \ast \Lambda_2' \ast \Lambda_0$ where $\Lambda_0$ is a clique and $\Lambda_1', \Lambda_2'$ two subgraphs which do not decompose as a join with a clique as a factor (and which are not empty according to Lemma \ref{lem:FlatRinX}). Notice that $\langle \Lambda_1 \ast \Lambda_2 \rangle$ decomposes as the Cartesian product $\langle \Lambda_1' \rangle \times \langle \Lambda_2' \rangle \times \langle \Lambda_0 \rangle$ where $\langle \Lambda_1' \rangle$ and $\langle \Lambda_2' \rangle$ are unbounded (because $\Lambda_1'$ and $\Lambda_2'$ are not cliques, so they contain at least two non-adjacent vertices) and essential (as a consequence of Lemma \ref{lem:Essential}). Moreover, every vertex of $\langle \Lambda_1 \ast \Lambda_2 \rangle$ is at distance at most $|V(\Lambda_0)| \leq \mathrm{clique}(\Gamma)$ from a vertex of $\langle \Lambda_1' \ast \Lambda_2' \rangle$, so that, as $M$ intersects non-trivially $\langle \Lambda_1 \ast \Lambda_2 \rangle$ (indeed, the intersection contains the flat rectangle $R$), necessarily $M^+$ has to intersect non-trivially $P:=\langle \Lambda_1' \ast \Lambda_2' \rangle$. Therefore, by noticing that
$$\begin{array}{lcl}| \mathcal{H}(M^+) \cap \mathcal{H}(P) | & \geq & | \mathcal{H}(M) \cap \mathcal{H}(P) | \geq | \mathcal{H}(R)| - |V(\Lambda_0)| \\ \\ & \geq & C+ \mathrm{clique}(\Gamma) - |V(\Lambda_0)| \geq C, \end{array}$$
it follows from Lemma \ref{lem:InterInclusion} that $P \subset M^+$. But, as noticed earlier, $\Lambda_1'$ and $\Lambda_2'$ are not cliques, so $\Lambda_1' \ast \Lambda_2'$ has to contain an induced square $\Lambda$. Therefore,
$$\langle \Lambda \rangle \subset \langle \Lambda_1' \ast \Lambda_2' \rangle = P \subset M^+,$$ 
concluding the proof of our lemma. 
\end{proof}

\noindent
Finally, our third preliminary lemma is an easy observation which we will use many times.

\begin{lemma}\label{lem:often}
Let $\Gamma$ be a finite simplicial graph, $\mathcal{G}$ a collection of groups indexed by $V(\Gamma)$, and $M \subset X(\Gamma,\mathcal{G})$ a gated Morse subgraph. If $\Lambda \subset \Gamma$ is an induced square which contains two non-adjacent vertices $a,b \in V(\Lambda)$ satisfying $\langle a,b \rangle \subset M$, then $\langle \Lambda \rangle \subset M$.
\end{lemma}

\begin{proof}
Notice that $\langle \Lambda \rangle$ decomposes as the Cartesian product of two unbounded leafless trees, one of them being $\langle a,b \rangle$. Consequently, it follows from Lemma \ref{lem:InterInclusion} that $\langle \Lambda \rangle \subset M$ whenever $\langle a,b \rangle \subset M$. 
\end{proof}

\noindent
From now on, we fix a finite simplicial graph $\Gamma$ and a collection of groups $\mathcal{G}$ indexed by $V(\Gamma)$.
The notion of \emph{slabbed subgraph} will be needed (and fundamental) in the proof of Theorem \ref{thm:eccentricsub}.

\begin{definition}
A \emph{slabbed subgraph} $(\Lambda, \mathcal{C})$ of $\Gamma$ is the data of an induced subgraph $\Lambda \leq \Gamma$ and a collection $\mathcal{C}$ of induced squares of $\Lambda$ such that, for every gated Morse subgraph $M \subset X(\Gamma, \mathcal{G})$, if $\langle \Xi \rangle \subset M$ for some $\Xi \in \mathcal{C}$ then $\langle \Lambda \rangle \subset M$. A square of $\mathcal{C}$ is referred to as a \emph{slab} of $\Lambda$.
\end{definition}

\noindent
The following definition will be also needed:

\begin{definition}
Let $\mathscr{C}=\{ (\Lambda_i, \mathcal{C}_i) \mid i \in I \}$ be a collection of slabbed subgraphs. The \emph{set of slabs of $\mathscr{C}$}, denoted by $\mathcal{S}(\mathscr{C})$, is the collection of all the slabs of its slabbed subgraphs. The \emph{oriented graph of $\mathscr{C}$}, denoted by $\mathscr{GC}$, is the graph whose whose vertex-set is $\mathscr{C}$ and whose oriented edges link a vertex $(\Lambda_i, \mathcal{C}_i)$ to a distinct vertex $(\Lambda_j, \mathcal{C}_j)$ if $\mathcal{C}_j$ contains a square having two opposite vertices in $\Lambda_i$. 
\end{definition}

\noindent
Before turning to the proof of Theorem \ref{thm:eccentricsub}, we register a few easy observations about collections of slabbed subgraphs.

\begin{lemma}\label{lem:GSedge}
Let $\mathscr{C}=\{ (\Lambda_i, \mathcal{C}_i) \mid i \in I \}$ be a collection of slabbed subgraphs. Fix two indices $i,j \in I$ and assume that $\mathscr{GC}$ has an oriented edge from $(\Lambda_i, \mathcal{C}_i)$ to $(\Lambda_j, \mathcal{C}_j)$. For every gated Morse subgraph $M \subset X(\Gamma, \mathcal{G})$, if $\langle \Lambda_i \rangle \subset M$ then $\langle \Lambda_j \rangle \subset M$.
\end{lemma}

\begin{proof}
By definition of $\mathscr{GC}$, there exists a square $\Xi \in \mathcal{C}_j$ having two opposite vertices in $\Lambda_i$. If $\langle \Lambda_i \rangle \subset M$, we deduce from Lemma \ref{lem:often} that $\langle \Xi \rangle \subset M$. By definition of a slabbed subgraph, we conclude that $\langle \Lambda_j \rangle \subset M$.
\end{proof}

\begin{lemma}\label{lem:GScycle}
Let $\mathscr{C}=\{ (\Lambda_i, \mathcal{C}_i) \mid i \in I \}$ be a collection of slabbed subgraphs, and let $\{ (\Lambda_j , \mathcal{C}_j) \mid j \in J\}$ denote the vertices of an oriented cycle in $\mathscr{GC}$. Then the implication
$$\left( \exists j \in J, \langle \Lambda_j \rangle \subset M \right) \Rightarrow \left\langle \bigcup\limits_{j \in J} \Lambda_j \right\rangle \subset M$$
holds for every gated Morse subgraph $M \subset X(\Gamma, \mathcal{G})$.
\end{lemma}

\begin{proof}
For convenience, we fix an enumeration $J= \{0, \ldots, n-1 \}$ such that $\mathscr{GC}$ contains an oriented edge from $(\Lambda_k, \mathcal{C}_k)$ to $(\Lambda_{k+1}, \mathcal{C}_{k+1})$ for every $k$ mod $n$. We also assume that $\langle \Lambda_0 \rangle \subset M$. 

\medskip \noindent
Fix some $g \in \langle \Lambda_0 \cup \cdots \cup \Lambda_{n-1} \rangle$. Write $g$ as a product $a_1 \cdots a_r$ such that, for every $1 \leq s \leq r$, there exists some index $i_s$ such that $a_s$ belongs to $\langle \Lambda_{i_s} \rangle$. Because $\langle \Lambda_{i_1} \rangle \subset M$ as a consequence of Lemma \ref{lem:GSedge}, we have
$$\langle \Lambda_{i_1} \rangle = a_1^{-1} \langle \Lambda_{i_1} \rangle \subset a_1^{-1} M.$$
Similarly, because $\langle \Lambda_{i_2} \rangle \subset a_1^{-1}M$ as a consequence of Lemma \ref{lem:GSedge}, we have
$$\langle \Lambda_{i_2} \rangle = a_2^{-1} \langle \Lambda_{i_2} \rangle \subset a_2^{-1}a_1^{-1} M.$$
By iterating the argument, it follows that $\langle \Lambda_{i_{r}} \rangle \subset a_{r}^{-1} \cdots a_1^{-1}M$. Since the vertex $1$ clearly belongs to $\langle \Lambda_{i_{r}} \rangle$, we conclude that
$$g = a_1 \cdots a_{r} \in M,$$
as desired.
\end{proof}

\noindent
Lemma \ref{lem:GScycle} allows us to simplify a collection of slabbed subgraphs as soon as its oriented graph contains an oriented cycle, as made explicit by the following statement:

\begin{cor}\label{cor:CollectionMinus}
Let $\mathscr{C}=\{ (\Lambda_i, \mathcal{C}_i) \mid i \in I \}$ be a collection of slabbed subgraphs, and let $\{ (\Lambda_j , \mathcal{C}_j) \mid j \in J\}$ denote the vertices of an oriented cycle in $\mathscr{GC}$. Then
$$\{ (\Lambda_i, \mathcal{C}_i) \mid i \in I \backslash J \} \cup \left\{ \left( \bigcup\limits_{j \in J} \Lambda_j, \bigcup\limits_{j \in J} \mathcal{C}_j \right) \right\}$$
is again a collection of slabbed subgraphs of $\Gamma$. Moreover, $\mathcal{S}(\mathscr{C})= \mathcal{S}(\mathscr{C}_0)$. 
\end{cor}

\begin{proof}
We have to show that, if $M \subset X(\Gamma, \mathcal{G})$ is a gated Morse subgraph and $\Xi$ a square which belongs to $\mathcal{C}_k$ for some $k \in J$, then $\langle \Xi \rangle \subset M$ implies $\left\langle \bigcup\limits_{j \in J} \Lambda_j \right\rangle \subset M$. 

\medskip \noindent
But we know by definition of a slabbed subgraph that $\langle \Xi \rangle \subset M$ implies that $\langle \Lambda_k \rangle \subset M$, so that the desired conclusion follows from Lemma \ref{lem:GScycle}. 
\end{proof}

\noindent
We are finally ready to turn to the proof of Theorem \ref{thm:eccentricsub}. 

\begin{proof}[Proof of Theorem \ref{thm:eccentricsub}.]
Let $M_0$ be a non-hyperbolic Morse subspace of $X(\Gamma, \mathcal{G})$. For now, we do not assume that $M_0$ is eccentric.
As a consequence of Lemmas \ref{lem:DistN} and \ref{lem:SquareEccentric}, there exists a non-hyperbolic Morse subgraph $M$ which contains a translate of $\langle \Xi \rangle$ for some induced square $\Xi$ of $\Gamma$ and which is contained into the $K$-neighborhood of $M_0$ for some constant $K$ depending only on $\Gamma$ and the Morse-gauge of $M_0$. From now on, we focus on the subspace $M$. 

\medskip \noindent
Set $\mathscr{C}_0 = \{ (\Lambda, \{ \Lambda \}) \mid \Lambda \leq \Gamma \ \text{induced square} \}$. Clearly, $\mathscr{C}_0$ is a collection of slabbed subgraphs of $\Gamma$. By applying Corollary \ref{cor:CollectionMinus} to $\mathcal{C}_0$ iteratively as many times as possible, one gets after finitely many steps (as the number of slabbed subgraphs decreases when applying Corollary \ref{cor:CollectionMinus}) a new collection of slabbed subgraphs $\mathscr{C}$. By construction, $\mathscr{GC}$ does not contain any oriented cycle. Moreover, $\mathcal{S}(\mathscr{C}_0)= \mathcal{S}(\mathscr{C})$.

\begin{claim}\label{claim:sink}
Assume that $(\Lambda, \mathcal{C})$ is a vertex-sink in $\mathscr{GC}$. Then $\Lambda$ is square-complete. 
\end{claim}

\noindent
If $\Lambda$ is not square-complete, then there exists an induced square $\Xi \leq \Gamma$ having two opposite vertices in $\Lambda$ but which is not included into $\Lambda$. Notice that, as $\Xi$ belongs to $\mathcal{S}(\mathscr{C}_0)= \mathcal{S}(\mathscr{C})$, there must exist $(\Lambda', \mathcal{C}') \in \mathscr{C}$ such that $\Xi \in \mathcal{C}'$. Moreover, $(\Lambda', \mathcal{C'})$ is distinct from $(\Lambda, \mathcal{C})$ as $\mathcal{C}'$ contains a square which is not included into $\Lambda$. Then $\mathscr{GC}$ contains an oriented edge from $(\Lambda, \mathcal{C})$ to $(\Lambda' , \mathcal{C}')$. Consequently, if $(\Lambda, \mathcal{C})$ is a vertex-sink in $\mathscr{GC}$, $\Lambda$ has to be square-complete, concluding the proof of our claim.

\medskip \noindent
Up to translating $M$, we will suppose without loss of generality that $M$ contains $\langle \Xi \rangle$. Because $\Xi \in \mathcal{S}(\mathscr{C}_0)= \mathcal{S}(\mathscr{C})$, there must exist some $(\Lambda_0, \mathcal{C}_0) \in \mathscr{C}$ such that $\Xi \in \mathcal{C}_0$. Notice that, as $\langle \Xi \rangle \subset M$ and $\Xi \in \mathcal{C}_0$, necessarily $\langle \Lambda_0 \rangle \subset M$ by definition of a slabbed subgraph.

\medskip \noindent
Because $\mathscr{GC}$ does not contain any oriented cycle, there must exist a vertex-sink $(\Lambda, \mathcal{C})$ such that $\mathscr{GC}$ contains an oriented path from $(\Lambda_0, \mathcal{C}_0)$ to $(\Lambda, \mathcal{C})$. Notice that, as a consequence of Lemma \ref{lem:GSedge} and of the existence of an oriented path from $(\Lambda_0, \mathcal{C}_0)$ to $(\Lambda, \mathcal{C})$, necessarily $\langle \Lambda \rangle \subset M$. Next, by combining Claim \ref{claim:sink} with Proposition \ref{prop:ParabolicMorse}, we know that $\Lambda$ is square-complete and that $\langle \Lambda \rangle$ is a Morse subgraph of $X(\Gamma, \mathcal{G})$. 

\medskip \noindent
So far, we have proved the following statement:

\begin{fact}\label{fact:ForConverse}
Let $\Gamma$ be a finite simplicial graph and $\mathcal{G}$ a collection of finite groups indexed by $V(\Gamma)$. If $M_0 \subset \Gamma \mathcal{G}$ is a non-hyperbolic Morse subspace, then there exist some element $g \in \Gamma \mathcal{G}$ and some square-complete graph $\Lambda \subset \Gamma$ which contains an induced square such that $g \langle \Lambda \rangle$ lies in the $K$-neighborhood of $M_0$ where $K$ depends only on $\Gamma$ and the Morse gauge of $M_0$. 
\end{fact}

\noindent
Now, assume that $M_0$ is not only a non-hyperbolic Morse subspace but an eccentric subspace. A fortiori, $M$ has to be a eccentric subspace as well. Because our subgraph $\langle \Lambda \rangle$ is not hyperbolic as $\Lambda$ contains an induced square, it follows from the definition of an eccentric subspace that the Hausdorff distance between $M$ and $\langle \Lambda \rangle$ is finite and depends only on the Morse-gauge of $\langle \Lambda \rangle$ and the eccentric-gauge of $M$. 

\medskip \noindent
In particular, $\langle \Lambda \rangle$ is an eccentric subspace in its own right. The desired conclusion follows from Lemma \ref{lem:ParabolicEccentric}.

\medskip \noindent
Thus, we have proved the implication of our theorem. Actually, we have proved a slightly stronger statement, which we record for future use:

\begin{fact}\label{fact:Improvement}
Let $\Gamma$ be a finite simplicial graph and $\mathcal{G}$ a collection of finite groups indexed by $V(\Gamma)$. If $M \subset \Gamma \mathcal{G}$ is eccentric, there exist an element $g \in \Gamma \mathcal{G}$ and a minsquare subgraph $\Lambda$ of $\Gamma$ such that the Hausdorff distance between $M$ and $g \langle \Lambda \rangle$ is bounded above by a finite constant depending only on $\Gamma$ and the Morse- and eccentric-gauges of $M$. 
\end{fact}

\noindent
Conversely, let $\Lambda \leq \Gamma$ be a minsquare subgraph. We want to show that $\langle \Lambda \rangle$ is an eccentric subspace. So let $Y \subset \langle \Lambda \rangle$ be a Morse subgraph which is not hyperbolic. As a consequence of Fact \ref{fact:ForConverse}, there must exist a square-complete subgraph $\Xi \leq \Gamma$ (which contains at least one induced square) and an element $g \in \Gamma \mathcal{G}$ such that $g \langle \Xi \rangle$ is contained into the $K$-neighborhood of $Y$, where $K$ is a constant which depends only on $\Gamma$ and the Morse-gauge of $Y$. It follows from Lemma \ref{lem:ParabolicHausdorff} that $\Xi$ decomposes as a join $\Xi_0 \ast \Xi_1$, where $\Xi_1$ is complete and where $\Xi_0$ is not the star of one of its vertices, and that $\Xi_0 \subset \Lambda$. As $\Lambda \cap \Xi$ is square-complete and contains an induced square (indeed, any induced square of $\Xi$ must be contained into $\Xi_0$, which is included into $\Lambda$), we deduce from the fact that $\Lambda$ is a minsquare subgraph of $\Gamma$ that $\Lambda= \Xi \cap \Lambda$. Notice that, since we have  $\Xi_0 \subset \Lambda \subset \Xi$, necessarily the Hausdorff distance between $\langle \Lambda \rangle$ and $\langle \Xi \rangle$ is at most $\mathrm{clique}(\Gamma)$. Therefore, we have:
$$\langle \Lambda \rangle = h\langle \Lambda \rangle \subset h \langle \Xi \rangle^{+ \mathrm{clique}(\Gamma)} \subset g \langle \Xi \rangle^{K+3 \mathrm{clique}(\Gamma)} \subset Y^{+2K+ 3 \mathrm{clique}(\Gamma)}$$
where $h \in \langle \Lambda \rangle$ is the element given by Lemma \ref{lem:ParabolicHausdorff}. Thus, we have proved that the Hausdorff distance between $Y$ and $\langle \Lambda \rangle$ is at most $2K+ 3\mathrm{clique}(\Gamma)$, which is a constant depend only on $\Gamma$ and the Morse-gauge of $Y$. We conclude that $\langle \Lambda \rangle$ is an eccentric subspace, as desired.
\end{proof}

\begin{remark}
In order to prove Theorem \ref{thm:eccentricsub}, we first showed that, given an eccentric subspace $M \subset X(\Gamma, \mathcal{G})$, up to finite Hausdorff distance and translation $M$ contains $\langle \Lambda \rangle$ for some induced square $\Lambda \leq \Gamma$ (Lemma \ref{lem:SquareEccentric}). Next, and it was the difficult part of the proof, we showed that the ``Morse hull'' of $\langle \Lambda \rangle$ is again a parabolic subgroup (and we determined it explicitly). In this remark, we would like to emphasize that the ``Morse hull'' of a parabolic subgroup is not always a parabolic group as well. For instance, let $\Gamma$ be the second graph given by Figure \ref{Five} below. Let $\Lambda$ denote the subgraph of $\Gamma$ obtained by removing the single vertex of degree two having two adjacent orange edges, and $\Xi$ the orange square. Because $\Lambda \cap \Xi$ contains two non-adjacent vertices, it follows from Lemma~\ref{lem:often} that any gated Morse subgraph containing $\langle \Lambda \rangle$ has to contain $\langle \Xi \rangle$ as well. The same holds for $g \langle \Xi \rangle$ if $g \in \langle \Lambda \rangle$. Now, it can be proved (by decomposing $X(\Gamma, \mathcal{G})$ as a tree of spaces whose vertex-spaces are copies of $\langle \Lambda \rangle$ and $\langle \Xi \rangle$, and by using Proposition \ref{prop:contracting}) that the union $\langle \Lambda \rangle \cup \bigcup\limits_{g \in \langle \Lambda \rangle} g \langle \Xi \rangle$ defines a gated Morse subgraph of $X(\Gamma, \mathcal{G})$, defining the ``Morse hull'' of $\langle \Lambda \rangle$. But it is not at finite Hausdorff distance from a parabolic subgroup. 
\end{remark}

\noindent
We conclude this section by proving some of the easy consequences of Theorem \ref{thm:eccentricsub} mentioned in the introduction.

\begin{proof}[Proof of Theorem \ref{thm:IntroMain}.]
Let $\Lambda_1$ be a minsquare subgraph of $\Gamma_1$. According to Theorem~\ref{thm:eccentricsub}, $\langle \Lambda_1 \rangle$ is an eccentric subspace of $\Gamma_1 \mathcal{G}_1$, so that $\Phi(\langle \Lambda_1 \rangle)$ must be an eccentric subspace of $\Gamma_2 \mathcal{G}_2$. Once again according to Theorem \ref{thm:eccentricsub}, there exist an element $g \in \Gamma_2 \mathcal{G}_2$ and a minsquare subgraph $\Lambda_2$ of $\Gamma_2$ such that $\Phi ( \langle \Lambda_1 \rangle)$ is at finite Hausdorff distance from $g \langle \Lambda_2 \rangle$, concluding the proof. 
\end{proof}

\begin{proof}[Proof of Proposition \ref{thm:IntroQIinv}.]
Assume that $\Gamma_1$ is a minsquare graph. As a consequence of Theorem \ref{thm:eccentricsub}, any eccentric subspace of $\Gamma_1 \mathcal{G}_1$ must be quasi-dense, so the same must be true for $\Gamma_2 \mathcal{G}_2$. Therefore, if we fix a minsquare subgraph $\Lambda_2$ in $\Gamma_2$, then $\langle \Lambda_2 \rangle$ must be quasi-dense in $\Gamma_2 \mathcal{G}_2$. The desired conclusion follows from Corollary \ref{cor:ParabolicHausdorff}.

\medskip \noindent
Next, assume that $\Gamma_1$ contains an induced square all of whose vertices are labelled by $\mathbb{Z}/2 \mathbb{Z}$'s and which is square complete. As a consequence of Theorem \ref{thm:eccentricsub}, $\Gamma_1 \mathcal{G}_1$ contains an eccentric subspace quasi-isometric to $\mathbb{Z}^2$, so that the same must be true for $\Gamma_2 \mathcal{G}_2$. By applying Theorem \ref{thm:eccentricsub} once again, it follows that there exists a minsquare subgraph $\Lambda$ in $\Gamma_2$ such that $\langle \Lambda \rangle$ is quasi-isometric to $\mathbb{Z}^2$. As a consequence of Proposition \ref{prop:hyp}, $\Lambda$ necessarily contains an induced square $\Xi$. If one of the vertices of $\Xi$ is labelled by a group of cardinality $>2$, then $\langle \Xi \rangle$ is quasi-isometric to $\mathbb{F}_2 \times \mathbb{Z}$ or $\mathbb{F}_2 \times \mathbb{F}_2$, which is not possible (since $\mathbb{Z}^2$ has polynomial growth). Therefore, $\Xi$ is an induced square all of whose vertices are labelled by $\mathbb{Z}/2\mathbb{Z}$'s. Next, because $\langle \Xi \rangle \subset \langle \Lambda \rangle$ and because $\langle \Xi \rangle$ and $\langle \Lambda \rangle$ are both quasi-isometric to $\mathbb{Z}^2$, necessarily the Hausdorff distance between $\langle \Lambda \rangle$ and $\langle \Xi \rangle$ must be finite. We conclude from Corollary \ref{cor:ParabolicHausdorff} that $\Lambda= \Xi$, so that the square $\Xi$ must be square-complete.
\end{proof}

\section{A few words about relative hyperbolicity}\label{section:RH}

\noindent
This section is dedicated to the proof of Theorem \ref{thm:IntroRH}. We begin by extracting a few statements from the study of relatively hyperbolic graph products of groups in \cite{Qm}. For information about relatively hyperbolic groups in full generality, we refer the reader to \cite{HruskaRH, OsinRelativeHyp}.

\medskip \noindent
In order to motivate the general criterion determining when a graph product of groups is relatively hyperbolic, we begin by mentioning the following sufficient condition:

\begin{prop}\label{prop:RH}
Let $\Gamma$ be a finite simplicial graph and $\mathcal{G}$ a collection of finite groups indexed by $V(\Gamma)$. Fix a collection $\mathcal{C}$ of subgraphs of $\Gamma$, and assume that:
\begin{itemize}
	\item the intersection between any two subgraphs of $\mathcal{C}$ is complete;
	\item any induced square of $\Gamma$ is contained into a subgraph of $\mathcal{C}$;
	\item for every $\Lambda \in \mathcal{C}$ and $u \in V(\Gamma) \backslash V(\Lambda)$, $\mathrm{link}(u) \cap \Lambda$ is complete.
\end{itemize}
Then $\Gamma \mathcal{G}$ is hyperbolic relative to $\mathcal{H}= \{ \langle \Lambda \rangle \mid \Lambda \in \mathcal{C} \}$. 
\end{prop}

\begin{proof}
Let $\mathcal{C}^+$ denote the collection of subgraphs of $\Gamma$ obtained from $\mathcal{C}$ by adding the singletons of the vertices of $\Gamma$ which do not belong to a subgraph of $\mathcal{C}$. Then \cite[Proposition 8.37]{Qm} applies, showing that $\Gamma \mathcal{G}$ is hyperbolic relative to $\mathcal{H}^+= \{ \langle \Lambda \rangle \mid \Lambda \in \mathcal{C}^+ \}$. Because $\mathcal{H}$ is obtained from $\mathcal{H}$ by removing finite subgroups, we conclude that $\Gamma \mathcal{G}$ is hyperbolic relative to $\mathcal{H}$ as well. 
\end{proof}

\noindent
Among all the collections of subgraphs satisfying the assumptions of Proposition \ref{prop:RH}, there is one ``minimal'' choice, namely:

\begin{definition}\label{decomposition}
Let $\Gamma$ be a finite graph. For every subgraph $\Lambda \subset \Gamma$, let $\mathrm{cp}(\Lambda)$ denote the subgraph of $\Gamma$ generated by $\Lambda$ and the vertices $v \in \Gamma$ such that $\mathrm{link}(v) \cap \Lambda$ is not complete. Now, define the collection of subgraphs $\mathfrak{J}^n(\Gamma)$ of $\Gamma$ by induction in the following way:
\begin{itemize}
	\item $\mathfrak{J}^0(\Gamma)$ is the collection of all the induced squares of $\Gamma$;
	\item if $C_1, \ldots, C_k$ denote the connected components of the graph whose set of vertices is $\mathfrak{J}^n(\Gamma)$ and whose edges link two subgraphs with non-complete intersection, we set $\mathfrak{J}^{n+1}(\Gamma)= \left( \mathrm{cp} \left( \bigcup\limits_{ \Lambda \in C_1} \Lambda \right), \ldots , \mathrm{cp} \left( \bigcup\limits_{\Lambda \in C_k} \Lambda \right) \right)$.
\end{itemize}
Because $\Gamma$ is finite, the sequence $(\mathfrak{J}^n(\Gamma))$ must be eventually constant to some collection $\mathfrak{J}^{\infty}(\Gamma)$.
\end{definition}

\noindent
Thanks to Proposition \ref{prop:RH}, it is not difficult to shows that this collection of subgraphs turns out to define a ``minimal'' collection of peripheral subgroups. More precisely:

\begin{thm}\label{thm:RHGP}
Let $\Gamma$ be a finite simplicial graph and $\mathcal{G}$ a collection of finite groups indexed by $V(\Gamma)$. Then $\Gamma \mathcal{G}$ is hyperbolic relative to $\mathcal{H}=\{ \langle \Lambda \rangle \mid \Lambda \in \mathfrak{J}^\infty(\Gamma) \}$. Moreover, if $\Gamma \mathcal{G}$ is hyperbolic relative to a collection of groups $\mathcal{K}$, then every subgroup of $\mathcal{H}$ must be contained into a conjugate of a subgroup of $\mathcal{K}$. 
\end{thm}

\begin{proof}
As a consequence of Proposition \ref{prop:RH}, it is clear that $\Gamma \mathcal{G}$ is hyperbolic relative to $\mathcal{H}$. Now, assume that $\Gamma \mathcal{H}$ is hyperbolic relative to another collection of subgroups $\mathcal{K}$. The fact that any subgroup of $\mathcal{H}$ must be contained into a conjugate of a subgroup of $\mathcal{K}$ is a consequence of the following three observations:
\begin{itemize}
	\item[(i)] If $\Lambda \leq \Gamma$ is an induced square, then $\langle \Lambda \rangle$ is contained into a conjugate of a subgroup of $\mathcal{K}$.
	\item[(ii)] Let $\Lambda_1, \Lambda_2 \leq \Gamma$ be two subgraphs such that $\Lambda_1 \cap \Lambda_2$ is not complete. Assume that there exist two conjugates $K_1$ and $K_2$ of subgroups of $\mathcal{K}$ such that $\langle \Lambda_1 \rangle \leq K_1$ and $\langle \Lambda_2 \rangle \leq K_2$. Then $K_1=K_2$. 
	\item[(iii)] Let $\Lambda \leq \Gamma$ be a subgraph and $u \in V(\Gamma)$ be a vertex such that $\mathrm{link}(u) \cap \Lambda$ is not complete. If there exists a conjugate $K$ of a subgroup of $\mathcal{K}$ containing $\langle \Lambda \rangle$, then $\langle \Lambda,u \rangle \subset K$.
\end{itemize}
The first observation follows from the fact that any subgroup isomorphic to a direct product of two infinite groups has to be included into a peripheral subgroup \cite[Theorems 4.16 and 4.19]{OsinRelativeHyp}. The second observation follows from the fact that peripheral subgroups define an almost malnormal collection \cite[Theorems 1.4 and 1.5]{OsinRelativeHyp}. This also implies the third observation because the intersection $K \cap uKu^{-1}$ contains the infinite subgroup $\langle \mathrm{link}(u) \rangle \cap \Lambda \rangle$. 
\end{proof}

\noindent
For instance, the right-angled Coxeter groups defined in Examples \ref{ex:five} and \ref{ex:NotQT} are not relatively hyperbolic, but those defined in Remark \ref{Charney} are relatively hyperbolic. In Figure~\ref{Intro}, only the fourth and seventh graph products are not relatively hyperbolic. 

\medskip \noindent
A consequence of Theorem \ref{thm:RHGP}, which will be useful in the proof of Theorem \ref{thm:IntroRH}, is the following:

\begin{cor}\label{cor:RH}
Let $\Gamma$ be a finite simplicial graph and $\mathcal{G}$ a collection of finite groups indexed by $V(\Gamma)$. Assume that $\Gamma \mathcal{G}$ is hyperbolic relative to a collection $\mathcal{K}$ of groups which are not relatively hyperbolic. Then each subgroup of $\mathcal{K}$ is a conjugate of a $\langle \Lambda \rangle$ for some $\Lambda \in \mathfrak{J}^\infty(\Gamma)$. 
\end{cor}

\begin{proof}
Fix a $K \in \mathcal{K}$. Because, $K$ is not relatively hyperbolic, it follows from \cite[Theorem 1.8]{DrutuSapir} and Theorem \ref{thm:RHGP} that $K \subset g \langle \Lambda \rangle g^{-1}$ for some $g \in \Gamma \mathcal{G}$ and $\Lambda \in \mathfrak{J}^\infty(\Gamma)$. But we also know from Theorem \ref{thm:RHGP} that $g \langle \Lambda \rangle g^{-1} \subset hK'h^{-1}$ for some $h \in \Gamma \mathcal{G}$ and $K' \in \mathcal{K}$. The inclusion $K \subset hK'h^{-1}$ implies that $K=hK'h^{-1}$ \cite[Theorems 1.4 and 1.5]{OsinRelativeHyp}, hence $K=g \langle \Lambda \rangle g^{-1}$, as desired.
\end{proof}

\noindent
The quasi-isometric rigidity contained in Theorem \ref{thm:IntroRH} comes from the following statement:

\begin{thm}\label{thm:BDM}\emph{\cite{BDM}}
Let $G$ be a finitely generated group hyperbolic relative to a finite collection of finitely generated subgroups $\mathcal{G}$ for which each $G \in \mathcal{G}$ is not relatively hyperbolic. Fix a finitely generated group $H$ and assume that there exists a quasi-isometry $\Phi : G \to H$. Then $H$ hyperbolic relative to a finite collection of finitely generated subgroups $\mathcal{H}$ such that $\Phi$ sends each subgroup in $\mathcal{H}$ at finite Hausdorff distance from one of the subgroups in $\mathcal{G}$.
\end{thm}

\noindent
We are now ready to prove Theorem \ref{thm:IntroRH}.

\begin{proof}[Proof of Theorem \ref{thm:IntroRH}.]
If $\Lambda \in \mathfrak{J}^\infty(\Gamma_1)$, then it follows from Theorem \ref{thm:RHGP} and from the construction of $\mathfrak{J}^\infty (\Gamma_1)$ that the graph product $\langle \Lambda \rangle$ is not relatively hyperbolic. As $\Gamma_1 \mathcal{G}_1$ is hyperbolic relative to $\mathcal{H}= \{ \langle \Lambda \rangle \mid \Lambda \in \mathfrak{J}^\infty(\Gamma_1) \}$, it follows from Theorem \ref{thm:BDM} that $\Gamma_2 \mathcal{G}_2$ is hyperbolic relative to some collection of subgroups $\mathcal{K}$ such that $\Phi$ sends each subgroup in $\mathcal{H}$ at finite Hausdorff distance from one of the subgroups in $\mathcal{K}$. Notice that, as being relatively hyperbolic is invariant under quasi-isometries \cite{RHandQI}, the subgroups of $\mathcal{K}$ are not relatively hyperbolic. We conclude from Corollary \ref{cor:RH} and Theorem \ref{thm:RHGP} that, if $\Lambda_1 \in \mathfrak{J}^\infty(\Gamma_1)$, there exist an element $g \in \Gamma_2 \mathcal{G}_2$ and a subgraph $\Lambda_2 \in \mathfrak{J}^\infty(\Gamma_2)$ such that $\Phi$ sends $\langle \Lambda_1 \rangle$ at finite Hausdorff distance from $g \langle \Lambda_2 \rangle$, as desired.
\end{proof}

\section{Electrification, quasi-isometries, hyperbolicity}\label{section:elec}

\subsection{The electrification}\label{subsection:elec}

\noindent
Now that we know that some subgroups in graph products of finite groups are quasi-isometrically rigid, it is natural to introduce a space which essentially encodes the way these subgroups are organised inside the entire group.

\begin{definition}
Let $\Gamma$ be simplicial graph and $\mathcal{G}$ a collection of groups indexed by $V(\Gamma)$. The \emph{electrification} $\mathbb{E}(\Gamma, \mathcal{G})$ is the Cayley graph of $\Gamma \mathcal{G}$ with respect to the generating set obtained from the union of the vertex-groups and $\bigcup\limits_{\Lambda \subset \Gamma \ \text{minsquare}} \langle \Lambda \rangle$. 
\end{definition}

\noindent
An equivalent point of view is to define $\mathbb{E}(\Gamma, \mathcal{G})$ as the \emph{cone-off} of $X(\Gamma, \mathcal{G})$ over $\mathcal{C}=\{ g \langle \Lambda \rangle \mid g \in \Gamma \mathcal{G} , \Lambda \leq \Gamma \ \text{minsquare} \}$, i.e., as the graph obtained from $X(\Gamma, \mathcal{G})$ by adding an edge between two vertices whenever they belong to a subgraph of $\mathcal{C}$. 

\medskip \noindent
The interesting point is that the quasi-isometric rigidity of eccentric subgroups implies that the electrification defines a quasi-isometric invariant of the group. More precisely:

\begin{prop}\label{prop:ElecQI}
Let $\Gamma_1, \Gamma_2$ be two finite simplicial graphs and $\mathcal{G}_1, \mathcal{G}_2$ two collections of finite groups indexed by $V(\Gamma_1), V(\Gamma_2)$ respectively. Any quasi-isometry $\Gamma_1 \mathcal{G}_1 \to \Gamma_2 \mathcal{G}_2$ induces a quasi-isometry $\mathbb{E}(\Gamma_1, \mathcal{G}_1) \to \mathbb{E} (\Gamma_2, \mathcal{G}_2)$. 
\end{prop}

\noindent
The proposition will be an immediate consequence of Theorem \ref{thm:eccentricsub} and the following two lemmas (well-known by the experts).

\begin{lemma}\label{lem:QIone}
Let $X,Y$ be two graphs, $f : X \to Y$ a quasi-isometry and $\mathcal{P}$ a collection of subspaces of $X$. Let $\dot{X}$ denote the cone-off of $X$ over $\mathcal{P}$ and $\dot{Y}$ the cone-off of $Y$ over the image of $\mathcal{P}$ under $f$. Then $f$ induces a quasi-isometry $\dot{X} \to \dot{Y}$.
\end{lemma}

\begin{proof}
Fix two constants $A>0$ and $B \geq 0$ such that
$$d_Y(f(x),f(y)) \leq A \cdot d_X(x,y)+B$$
for every $x,y \in X$. Let $x,y \in X$ be two vertices and let $x_0, \ldots, x_n $ be the vertices of some geodesic in $\dot{X}$ between $x$ and $y$. For every $1 \leq i \leq n$, either $x_i$ and $x_{i+1}$ belong to a common subspace of $\mathcal{P}$, so that $d_{\dot{Y}}(f(x_i),f(x_{i+1})) \leq 1$; or $x_i$ and $x_{i+1}$ are adjacent in $X$, so that 
$$d_{\dot{Y}}(f(x_i),f(x_{i+1})) \leq d_Y(f(x_i),f(x_{i+1})) \leq A \cdot d_X(x_i,x_{i+1})+B= A+B.$$ 
Therefore,
$$\begin{array}{lcl} d_{\dot{Y}}(f(x),f(y)) & \leq & \displaystyle \sum\limits_{i=0}^{n-1} d_{\dot{Y}}(f(x_i),f(x_{i+1})) \leq \max(1,A+B) \cdot n \\ \\ & = & \max(1,A+B) \cdot d_{\dot{X}}(x,y). \end{array}$$
We argue similarly for the other inequality. Let $x_0, \ldots, x_n$ be the vertices of some geodesic between $f(x)$ and $f(y)$ in $\dot{Y}$. Let $f^{-1}$ denote some quasi-inverse of $f$ and fix three constants $R>0$ and $K,S \geq 0$ such that 
$$ d_X \left( f^{-1}(x), f^{-1}(y) \right) \leq R \cdot d_Y(x,y) + S \ \text{and} \ d_X \left( z, f^{-1} f (z) \right) \leq K$$
for every $x,y \in Y$ and $z \in X$. Notice that

\begin{fact}
If $a \in f(P)$ for some $P \in \mathcal{P}$, then $d_X(f^{-1}(a),P) \leq K$.
\end{fact}

\noindent
Indeed, there exists some $b \in P$ such that $f(b)=a$, and 
$$d_X \left( f^{-1}(a),P \right) \leq d_X \left( f^{-1}(a),b \right) \leq d_X \left( f^{-1}(a),f^{-1}f(b) \right) + K=K,$$
which proves the fact. As a consequence, if for some $0 \leq i \leq n-1$ the vertices $x_i$ and $x_{i+1}$ belong to $f(P)$ for some $P \in \mathcal{P}$, then 
$$\begin{array}{lcl} d_{\dot{X}}\left( f^{-1}(x_i),f^{-1}(x_{i+1}) \right) & \leq & d_{X}\left( f^{-1}(x_i),f^{-1}(x_{i+1}) \right) \\ \\ & \leq & d_X \left( f^{-1}(x_i), P \right) +1+ d_X \left( f^{-1}(x_{i+1}, P \right) \leq 2K+1. \end{array}$$
Next, if the vertices $x_i$ and $x_{i+1}$ are adjacent in $\dot{Y}$, then 
$$d_{\dot{X}} \left( f^{-1}(x_i),f^{-1}(x_{i+1}) \right) \leq d_X \left( f^{-1}(x_i),f^{-1}(x_{i+1}) \right) \leq R \cdot d_Y(x_i,x_{i+1})+S=R+S.$$ 
Therefore,
$$\begin{array}{lcl} d_{\dot{X}}(x,y) & \leq & \displaystyle 2K+ d_{\dot{X}} \left( f^{-1}f(x),f^{-1}f(y) \right) \leq 2K+ \sum\limits_{i=0}^{n-1} d_{\dot{X}} \left( f^{-1}(x_i),f^{-1}(x_{i+1}) \right) \\ \\ & \leq & 2K+ \max (2K+1,R+S) \cdot n = 2K+ \max(2K+1,R+S) \cdot d_{\dot{Y}}(f(x),f(y)). \end{array}$$
Thus, we have proved that $f$ induces a quasi-isometric embedding $\dot{X} \to \dot{Y}$. It is clear that the image of $f$ is quasi-dense in $\dot{Y}$, which concludes the proof. 
\end{proof}

\begin{lemma}\label{lem:QItwo}
Let $X$ be a graph and $\mathcal{P}_1, \mathcal{P}_2$ two collections of subspaces of $X$. Assume that there exists some $D \geq 0$ such that, for every $P_1 \in \mathcal{P}_1$ (resp. $P_2 \in \mathcal{P}_2$), there exists $P_2 \in \mathcal{P}_2$ (resp. $P_1 \in \mathcal{P}_1$) such that the Hausdorff distance between $P_1$ and $P_2$ is $\leq D$. Let $X_1$ (resp. $X_2$) denote the cone-off of $X$ over $\mathcal{P}_1$ (resp. $\mathcal{P}_2$). The canonical map $X_1^{(0)} \to X_2^{(0)}$ induces a quasi-isometry $X_1 \to X_2$. 
\end{lemma}

\begin{proof}
Fix two vertices $a,b \in X$ and let $x_0, \ldots, x_n$ denote the vertices of a geodesic in $X_1$ between $a$ and $b$. For every $0 \leq i \leq n-1$, either $x_i$ and $x_{i+1}$ are adjacent in $X$, so that $d_{X_2}(x_i,x_{i+1})=1$; or $x_i$ and $x_{i+1}$ belong to a subspace of $\mathcal{P}_1$, so that $d_{X_2}(x_i,x_{i+1}) \leq 2D+1$. Consequently,
$$d_{X_2}(a,b) \leq \sum\limits_{i=0}^{n-1} d_{X_2}(x_i,x_{i+1}) \leq (2D+1) \cdot n = (2D+1) \cdot d_{X_1}(a,b).$$
A symmetric argument shows that $d_{X_1} \leq (2D+1) \cdot d_{X_2}$. Thus, $X_1$ and $X_2$ must be quasi-isometric.
\end{proof}

\begin{proof}[Proof of Proposition \ref{prop:ElecQI}.]
Fix a quasi-isometry $\Phi : \Gamma_1 \mathcal{G}_1 \to \Gamma_2 \mathcal{G}_2$. By definition, the electrification $\mathbb{E}(\Gamma_1, \mathcal{G}_1)$ is the cone-off of $X(\Gamma_1, \mathcal{G}_2)$ over $\mathcal{P}_1= \{ g \langle \Lambda \rangle \mid g \in \Gamma_1 \mathcal{G}_1, \Lambda \leq \Gamma_1 \ \text{minsquare} \}$. Notice that $\mathcal{P}_1$ contains only finitely many $\Gamma_1 \mathcal{G}_1$-orbits of subspaces, so that the subspaces of $\mathcal{P}_1$ are \emph{uniformly eccentric} (resp. \emph{uniformly Morse}), i.e., they share a common eccentric-gauge (resp. Morse-gauge). Necessarily, the same holds for $\Phi (\mathcal{P}_1)$, so that we deduce from Fact \ref{fact:Improvement} and Lemma \ref{lem:QItwo} that the cone-off of $X(\Gamma_2, \mathcal{G}_2)$ over $\Phi(\mathcal{P}_1)$ is quasi-isometric to the electrification $\mathbb{E}(\Gamma_2, \mathcal{G}_2)$. The desired conclusion follows from Lemma \ref{lem:QIone}.
\end{proof}

\subsection{When is the electrification hyperbolic?}\label{subsection:elechyp}

\noindent
As the electrification turns out to be a quasi-isometric invariant, it is natural to study its geometry. In this section, we focus on its hyperbolicity by proving the following proposition:

\begin{prop}\label{prop:ElecHyp}
Let $\Gamma$ be a finite simplicial graph and $\mathcal{G}$ a collection of groups indexed by $V(\Gamma)$. Then $\mathbb{E}(\Gamma, \mathcal{G})$ is hyperbolic if and only if every induced square of $\Gamma$ is included into some minsquare subgraph.
\end{prop}

\noindent
The hyperbolicity will be obtained from a criterion proved in \cite{Qm}. The following lemma, which we think to be of independent interest, will be used to prove non-hyperbolicity.

\begin{lemma}\label{lem:QIembed}
Let $\Gamma$ be a finite simplicial graph, $\mathcal{G}$ a collection of groups indexed by $V(\Gamma)$, and $\mathcal{S}$ be a collection of subgraphs of $\Gamma$. Let $\dot{X}$ denote the cone-off of $X(\Gamma, \mathcal{G})$ over $\{g \langle \Lambda \rangle \mid g \in \Gamma \mathcal{G}, \Lambda \in \mathcal{S} \}$. If $\Lambda \subset \Gamma$ is a subgraph whose intersection with any subgraph of $\mathcal{S}$ is either empty or complete, then $\langle \Lambda \rangle$ quasi-isometrically embeds into $\dot{X}$. 
\end{lemma}

\begin{proof}
Let $x,y \in \langle \Lambda \rangle$ be two vertices, and let $J_1, \ldots, J_n$ be a collection of pairwise non-transverse hyperplanes separating $x$ and $y$ which has maximal cardinality. Notice that, for every $1 \leq i \leq n-1$, the hyperplanes $J_i$ and $J_{i+1}$ are tangent, since otherwise there would be a new hyperplane separating $J_i$ and $J_{i+1}$ (as a consequence of Proposition~\ref{prop:BigProj}), contracting the maximality of our collection. Let $x_1, \ldots, x_r$ denote the vertices of some geodesic in $\mathbb{E}(\Gamma, \mathcal{G})$ between $x$ and $y$. 

\medskip \noindent
Suppose that there exists some $1 \leq i \leq n-1$ such that no $x_j$ belongs to the subspace delimited by $J_i$ and $J_{i+1}$. Let $x_a$ denote the last $x_j$ which does not belong to the sector delimited by $J_{i}$ containing $y$; by definition $x_{a+1}$ belongs to the sector delimited by $J_{i+1}$ containing $y$. Consequently, $J_i$ and $J_{i+1}$ separate $x_a$ and $x_{a+1}$. Since $d_{\dot{X}}(x_a,x_{a+1})=1$ but $d_{X}(x_a,x_{a+1}) \geq 2$, necessarily there exist $g \in \Gamma \mathcal{G}$ and $\Xi \in \mathcal{S}$ such that $x_a$ and $x_{a+1}$ both belong to $g \langle \Xi \rangle$. A fortiori, $J_i$ and $J_{i+1}$ intersect $g \langle \Xi \rangle$. Let $u$ and $v$ denote the two vertices of $\Gamma$ labelling $J_i$ and $J_{i+1}$ respectively. Because $J_i$ and $J_{i+1}$ cross $g \langle \Xi \rangle$, necessarily $u,v \in \Xi$; and because $J_i$ and $J_{i+1}$ separate $x,y \in \langle \Lambda \rangle$, necessarily $u,v \in \Lambda$. Thus, $\Lambda \cap \Xi$ contains $u$ and $v$, which we know to be non-adjacent as a consequence of Lemma \ref{lem:HypTransverseLabel} since they label two tangent hyperplanes. This contradicts our assumptions. 

\medskip \noindent
Thus, we have proved that, for every $1 \leq i \leq n-1$, there exists some $1 \leq j \leq r$ such that $x_j$ belongs to the subspace delimited by $J_i$ and $J_{i+1}$. It follows that $r \geq n+1$. 

\begin{claim}
We have $n \geq d_X(x,y) / \mathrm{clique}(\Gamma)$. 
\end{claim}

\noindent
Let $\mathcal{S}$ denote the collection of all the sectors containing $y$ but not $x$, partially ordered by the inclusion. Notice that $n$ coincides with the maximal cardinality of a chain in $\mathcal{S}$, and that any antichain in $\mathcal{S}$ has cardinality at most $\dim_\square(X)$. It follows from Dilworth's theorem that
$$d_X(x,y)= \# \mathcal{S} \leq \dim_\square(X) \cdot n,$$
concluding the proof of our claim.

\medskip \noindent
Finally, we have
$$\frac{1}{\mathrm{clique}(\Gamma)} \cdot d_X(x,y) \leq n \leq r-1= d_{\dot{X}}(x,y) \leq d_X(x,y),$$
concluding the proof.
\end{proof}

\begin{proof}[Proof of Proposition \ref{prop:ElecHyp}.]
If there exists some induced square $\Lambda \subset \Gamma$ which is not included into any minsquare subgraph, then, because a minsquare subgraph is necessarily square-complete, $\Lambda$ cannot intersect such a subgraph along at least two non-adjacent vertices. It follows from Lemma \ref{lem:QIembed} that $\langle \Lambda \rangle$ quasi-isometrically embeds into the electrification $\mathbb{E}(\Gamma, \mathcal{G})$. As $\langle \Lambda \rangle$ is isometric to a product of two unbounded trees, we conclude that $\mathbb{E}(\Gamma, \mathcal{G})$ cannot be hyperbolic.

\medskip \noindent
From now on, suppose that any induced square of $\Gamma$ is contained into some minsquare subgraph. Let $R \subset X$ be an $L$-thick flat rectangle where $L> \mathrm{clique}(\Gamma)$. According to Lemma \ref{lem:FlatRinX}, $R \subset g \langle \Lambda_1' \ast \Lambda_2' \rangle$ for some element $g \in \Gamma \mathcal{G}$ and some subgraphs $\Lambda_1',\Lambda_2'$. The join $\Lambda_1' \ast \Lambda_2'$ can be written as $\Lambda_1 \ast \Lambda_2 \ast \Lambda_0$ where $\Lambda_1 \subset \Lambda_1'$ and $\Lambda_2 \subset \Lambda_2'$ are two subgraphs without any vertex whose star covers all their vertices, and where $\Lambda_0$ is a complete subgraph. Because $R$ is $L$-thick with $L> \mathrm{clique}(\Gamma)$, both $\Lambda_1'$ and $\Lambda_2'$, a fortiori both $\Lambda_1$ and $\Lambda_2$, must contain two non-adjacent vertices, say $u_1,v_1$ and $u_2,v_2$ respectively. By assumption, there exists some minsquare subgraph $\Xi$ which contains the induced square defined by $u_1,u_2,v_1,v_2$. Now let $w$ be a vertex of $\Lambda_1$. By construction of $\Lambda_1$, there exists a vertex $w^* \in \Lambda_1$ which is not adjacent to $w$. By noticing that the two squares defined respectively by $u_1,u_2,v_1,v_2$ and $w,u_2,w^*,v_2$ share two non adjacent vertices, it follows that $w$ belongs to $\Xi$. Therefore, $\Lambda_1 \subset \Xi$, and one proves similarly that $\Lambda_2 \subset \Xi$. Consequently, the subgraph $\langle \Lambda_1 \ast \Lambda_2 \rangle$ has diameter one in $\mathbb{E}(\Gamma, \mathcal{G})$, and it follows that
$$g \langle \Lambda_1' \ast \Lambda_2' \rangle = g \langle \Lambda_0 \rangle \times g \langle \Lambda_1 \ast \Lambda_2 \rangle$$
has diameter at most $1+ \mathrm{clique}(\Gamma)$ in $\mathbb{E}(\Gamma, \mathcal{G})$. A fortiori, our flat rectangle $R$ has diameter at most $1+ \mathrm{clique}(\Gamma)$ in $\mathbb{E}(\Gamma, \mathcal{G})$.

\medskip \noindent
Thus, we have proved that any $(1+ \mathrm{clique}(\Gamma))$-thick flat rectangle of $X(\Gamma, \mathcal{G})$ has diameter at most $(1+ \mathrm{clique}(\Gamma))$ in $\mathbb{E}(\Gamma, \mathcal{G})$. We conclude from \cite[Proposition 8.38]{Qm} that $\mathbb{E}(\Gamma, \mathcal{G})$ is hyperbolic.
\end{proof}

\begin{remark}
It is worth noticing that, when proving Proposition \ref{prop:ElecHyp}, we have shown the following general statement: Let $\Gamma$ be a finite simplicial graph, $\mathcal{G}$ a collection of groups indexed by $V(\Gamma)$, and $\mathcal{S}$ a collection of subgraphs of $\Gamma$. If every join $\Lambda_1 \ast \Lambda_2 \leq \Gamma$, where $\Lambda_1$ and $\Lambda_2$ have no vertex whose star covers all their vertices, is included into some subgraph of $\mathcal{S}$, then the cone-off of $X(\Gamma, \mathcal{G})$ over $\{ g \langle \Lambda \rangle \mid g \in \Gamma \mathcal{G}, \Lambda \in \mathcal{S} \}$ is hyperbolic. However, such a criterion is not optimal. For instance, if $\Gamma$ is a square and if $\mathcal{S}$ is reduced to a single graph, namely the disjoint union of two opposite vertices of $\Gamma$, then the previous criterion does not apply but our cone-off is hyperbolic (and is quasi-isometric to a tree).
\end{remark}

\begin{figure}
\begin{center}
\includegraphics[scale=0.4]{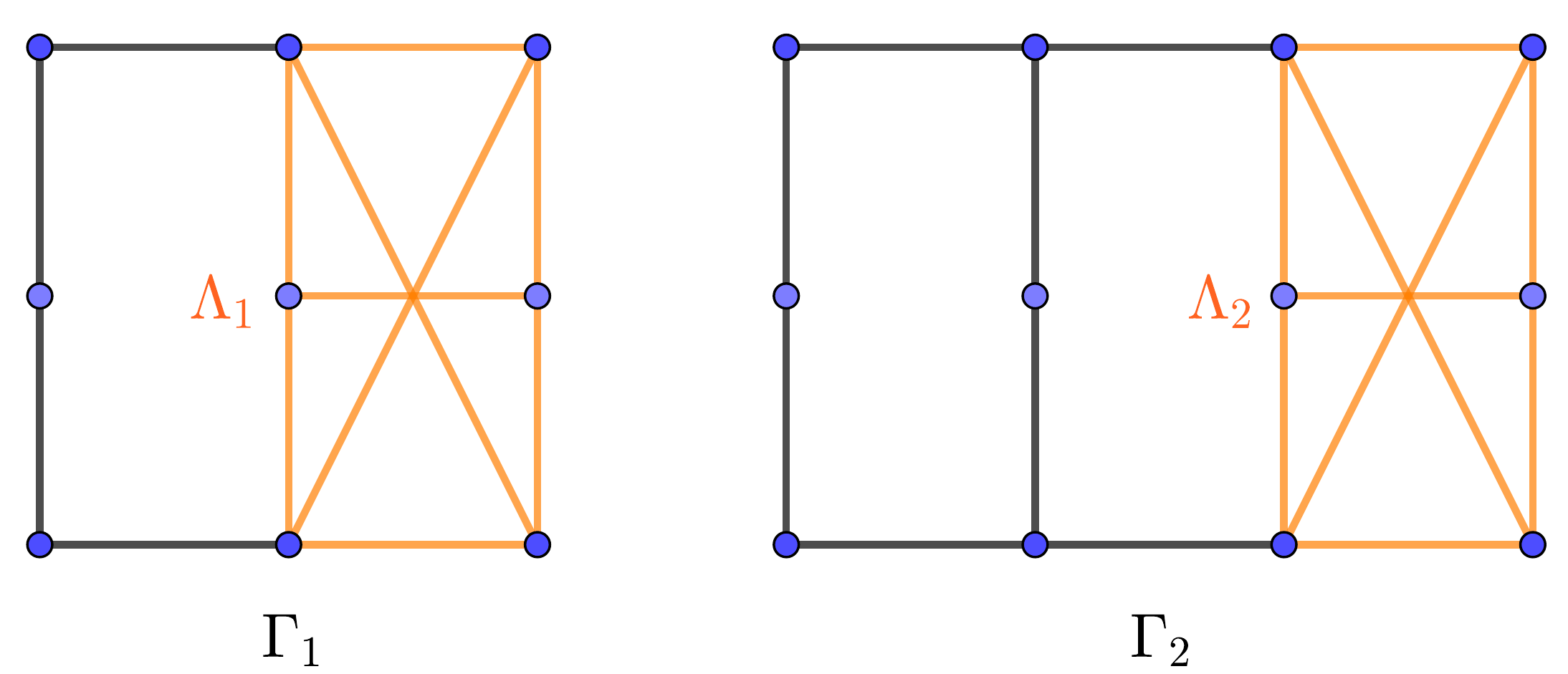}
\caption{Two relatively hyperbolic right-angled Coxeter groups which are not quasi-isometric.}
\label{Charney}
\end{center}
\end{figure}

\begin{remark}\label{remark:Charney}
In \cite[Section 4.2]{ContractingBoundary}, it is proved that the two right-angled Coxeter groups $C(\Gamma_1)$ and $C(\Gamma_2)$ defined by the graphs $\Gamma_1$ and $\Gamma_2$ of Figure \ref{Charney} are not quasi-isometric by looking at their contracting boundaries and by noticing that they are not homeomorphic. As an application of our results, it is possible to reprove this fact. 

\medskip \noindent
A sketch of proof goes as follows. The graphs $\Gamma_1$ and $\Gamma_2$ contain a unique minsquare subgraph, namely $\Lambda_1$ and $\Lambda_2$ respectively. It follows from Proposition \ref{prop:ElecHyp} that the two electrifications $\mathbb{E}(\Gamma_1)$ and $\mathbb{E}(\Gamma_2)$ are hyperbolic, but we claim that one is a quasi-tree and not the other one. Indeed, if $\Xi_2$ denote the left pentagon of $\Gamma_2$, then it follows from Lemma \ref{prop:ElecHyp} that $\langle \Xi_2 \rangle$ quasi-isometrically embed into the electrification. But $\langle \Xi_2 \rangle$ is quasi-isometric to $\mathbb{H}^2$, so that $\mathbb{E}(\Gamma_2)$ cannot be a quasi-tree. In order to show that $\mathbb{E}(\Gamma_1)$ is a quasi-tree, the idea is to decompose $C(\Gamma_1)$ as an amalgamated product of $\langle \Xi_1 \rangle$ (where $\Xi_1$ is the left pentagon of $\Gamma_1$) and $\langle \Lambda_1 \rangle$ over the virtually cyclic subgroup $\langle \Lambda_1 \cap \Xi_1 \rangle$, and to notice that the cone-off of $\langle \Xi_1 \rangle$ (which is quasi-isometric to a hyperbolic plane) over the cosets of $\langle \Lambda_1 \cap \Xi_1 \rangle$ is a quasi-tree. 

\medskip \noindent
However, our two right-angled Coxeter groups are relatively hyperbolic. More precisely, $C(\Gamma_1)$ (resp. $C(\Gamma_2)$) is hyperbolic relative to $\langle \Lambda_1 \rangle$ (resp. $\langle \Lambda_2 \rangle$). As we are mainly interested in groups which are not relatively hyperbolic (since the quasi-isometric rigidity of their peripheral subgroups is already known \cite{RHandQI}), we do not give further details here. 
\end{remark}

\section{Hyperbolicity of infinite-index Morse subgroups}\label{section:MorseSubHyp}

\noindent
As an other application of Theorem \ref{thm:eccentricsub}, we are able to determine precisely when a graph product of finite groups has all its infinite-index Morse subgroup hyperbolic. Namely:

\begin{thm}\label{thm:MorseAllHyp}
Let $\Gamma$ be a simplicial graph and $\mathcal{G}$ a collection of finite groups indexed by $V(\Gamma)$. The infinite-index Morse subgroups of $\Gamma \mathcal{G}$ are all hyperbolic if and only if $\Gamma$ is square-free or if it decomposes as the join of a minsquare subgraph and a complete graph.  
\end{thm}

\begin{proof}
If $\Gamma$ is square-free, then $\Gamma \mathcal{G}$ is hyperbolic according to Corollary \ref{cor:hyp}, so that Morse subgroups are quasiconvex and so are hyperbolic.

\medskip \noindent
Now, assume that $\Gamma$ decomposes as the join $\Gamma_0 \ast \Gamma_1$ of a minsquare subgraph $\Gamma_0$ and a complete graph $\Gamma_1$. Let $M \subset \Gamma \mathcal{G}$ be a non-hyperbolic Morse subgroup. As a consequence of Theorem \ref{thm:eccentricsub}, there exists a minsquare subgraph $\Lambda \subset \Gamma$ and an element $g \in \Gamma \mathcal{G}$ such that the Hausdorff distance between $g \langle \Lambda \rangle$ and $M$ in $X(\Gamma, \mathcal{G})$ is finite. Because $\Lambda$ contains an induced square and that any induced square in $\Gamma_0 \ast \Gamma_1$ must be included into $\Gamma_0$, it follows that $\Lambda \cap \Gamma_0$ contains an induced square, hence $\Lambda= \Gamma_0$ by definition of minsquare subgraphs. Thus, we have proved that the Hausdorff distance between $M$ and $g \langle \Gamma_0 \rangle$ is finite. As $X(\Gamma, \mathcal{G})$ decomposes as the Cartesian product of $g \langle \Gamma_0 \rangle$ and the prism $g \langle \Gamma_1 \rangle$, we conclude that $M$ is quasi-dense in the Cayley graph $X(\Gamma, \mathcal{G})$. Therefore, $M$ must be a finite-index subgroup.

\medskip \noindent
Conversely, suppose that the infinite-index Morse subgroups of $\Gamma \mathcal{G}$ are hyperbolic. If $\Gamma$ is square-free, there is nothing to prove, so suppose that $\Gamma$ contains at least one induced subgraph. Fix a minsquare subgraph $\Lambda \subset \Gamma$. As a consequence of Propositions \ref{prop:hyp} and~\ref{prop:ParabolicMorse}, $\langle \Lambda \rangle$ is a non-hyperbolic Morse subgroup. By assumption, $\langle \Lambda \rangle$ must be a finite-index subgroup of $\Gamma \mathcal{G}$. In other words, the Hausdorff distance between $\langle \Lambda \rangle$ and $\Gamma \mathcal{G}= \langle \Gamma \rangle$ is finite, so that the desired conclusion follows from Corollary \ref{cor:ParabolicHausdorff}. 
\end{proof}

\section{Examples}\label{section:ex}

\noindent
For simplicity, all the examples we give in this section are right-angled Coxeter groups, so that we do not need to label vertices with groups as they are all automatically labelled by $\mathbb{Z}/2 \mathbb{Z}$. We also emphasize that none of the right-angled Coxeter groups mentioned here (except in Example \ref{ex:PhiPsi}) is relatively hyperbolic.
\begin{figure}
\begin{center}
\includegraphics[scale=0.4]{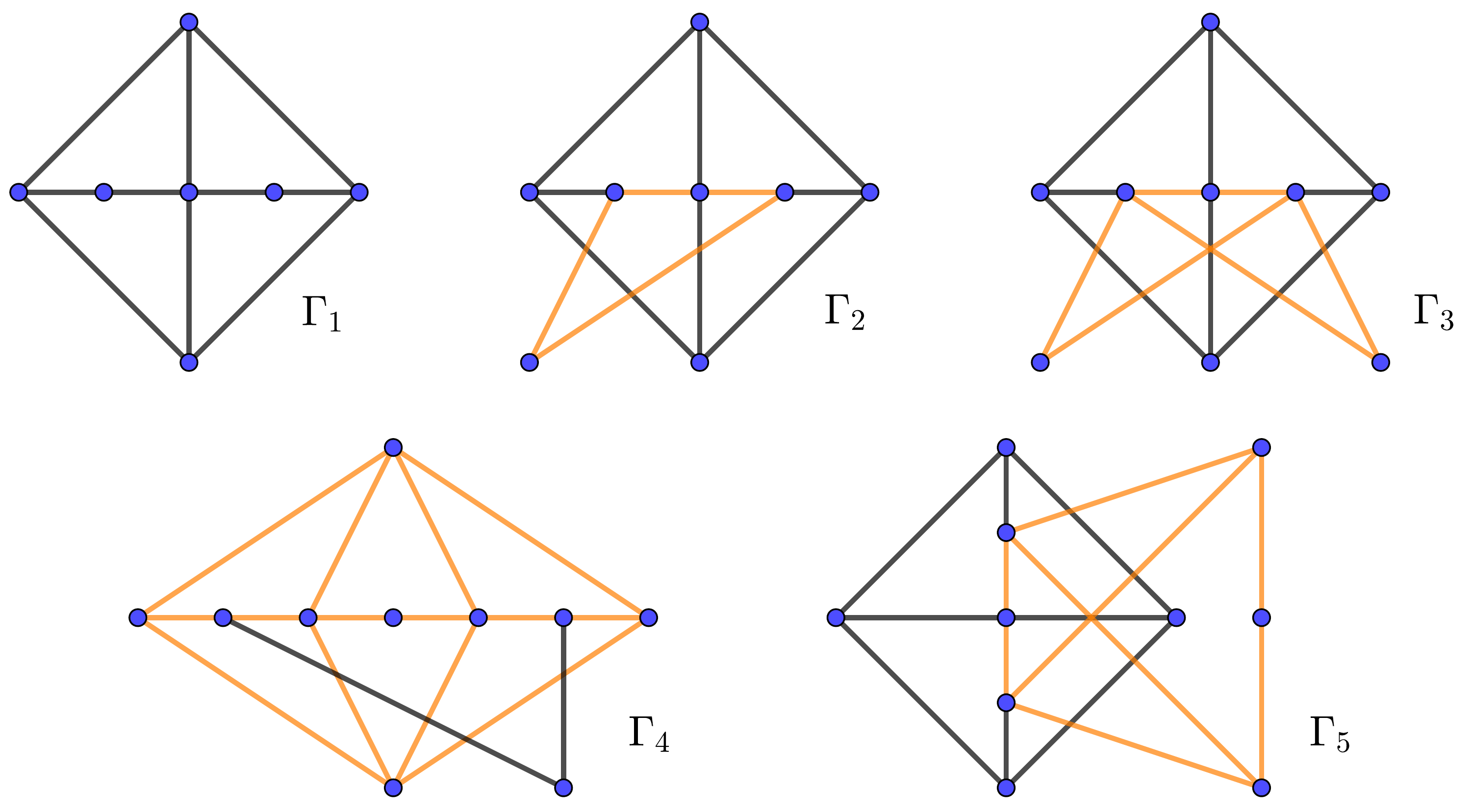}
\caption{Five graphs defining five pairwise non-quasi-isometric right-angled Coxeter groups. In orange are minsquare subgraphs.}
\label{Five}
\end{center}
\end{figure}

\begin{ex}\label{ex:five}
We claim that the right-angled Coxeter groups defined by the five graphs given by Figure \ref{Five} are pairwise non-quasi-isometric. By applying Theorem \ref{thm:eccentricsub} and Proposition \ref{prop:ElecHyp}, we find quasi-isometric invariants which allow us to distinguish all these groups:
\begin{itemize}
	\item $\Gamma_1$ is a minsquare graph, so any eccentric subspace in $C(\Gamma_1)$ must be quasi-dense and $\mathbb{E}(\Gamma_1)$ is bounded.
	\item The eccentric subspaces of $C(\Gamma_2)$ are quasi-isometric to $\mathbb{Z}^2$ and $\mathbb{E}(\Gamma_2)$ is not hyperbolic.
	\item The eccentric subspaces of $C(\Gamma_3)$ are quasi-isometric to $\mathbb{F}_2\times \mathbb{Z}$ and $\mathbb{E}(\Gamma_3)$ is not hyperbolic.
	\item The eccentric subspaces of $C(\Gamma_4)$ have superlinear divergence \cite{RACGdiv} and $\mathbb{E}(\Gamma_4)$ is hyperbolic.
	\item The eccentric subspaces of $C(\Gamma_5)$ have superlinear divergence \cite{RACGdiv} and $\mathbb{E}(\Gamma_5)$ is not hyperbolic.
\end{itemize}
\end{ex}

\begin{ex}\label{ex:NotQT}
A natural question is to ask whether the electrification may be hyperbolic but without being quasi-isometric to a tree. For hyperbolic graph products of finite groups, such examples clearly exist as the electrification turns out to coincide with the group itself. But it is not an interesting example. A more interesting example is given in Remark \ref{remark:Charney}, but the corresponding right-angled Coxeter group turns out to be relatively hyperbolic. Finding an example which is not relatively hyperbolic seems to be more delicate. An example is given by Figure \ref{NotQT}.

\medskip \noindent
Let $\Gamma$ be the graph given by Figure \ref{NotQT}. The minsquare subgraphs of $\Gamma$ are the five copies of $\Phi$. We deduce from Proposition \ref{prop:ElecHyp} that $\mathbb{E}(\Gamma)$ is hyperbolic. If $\Lambda$ denote the central pentagon of $\Gamma$, then it follows from Lemma \ref{lem:QIembed} that $\langle \Lambda \rangle$ embeds quasi-isometrically into $\mathbb{E}(\Gamma)$. As $\langle \Lambda \rangle$ is quasi-isometric to $\mathbb{H}^2$, we conclude that $\mathbb{E}(\Gamma)$ cannot be a quasi-tree.

\medskip \noindent
As a consequence, $C(\Gamma)$ is not quasi-isometric to any of the five right-angled Coxeter groups defined by the graphs of Figure \ref{Five}. The only non-trivial point to check is that the electrification $\mathbb{E}(\Gamma_4)$ is a quasi-tree. A strategy similar to that sketched in Remark~\ref{remark:Charney} can be applied. We do not give more details here.
\begin{figure}
\begin{center}
\includegraphics[scale=0.4]{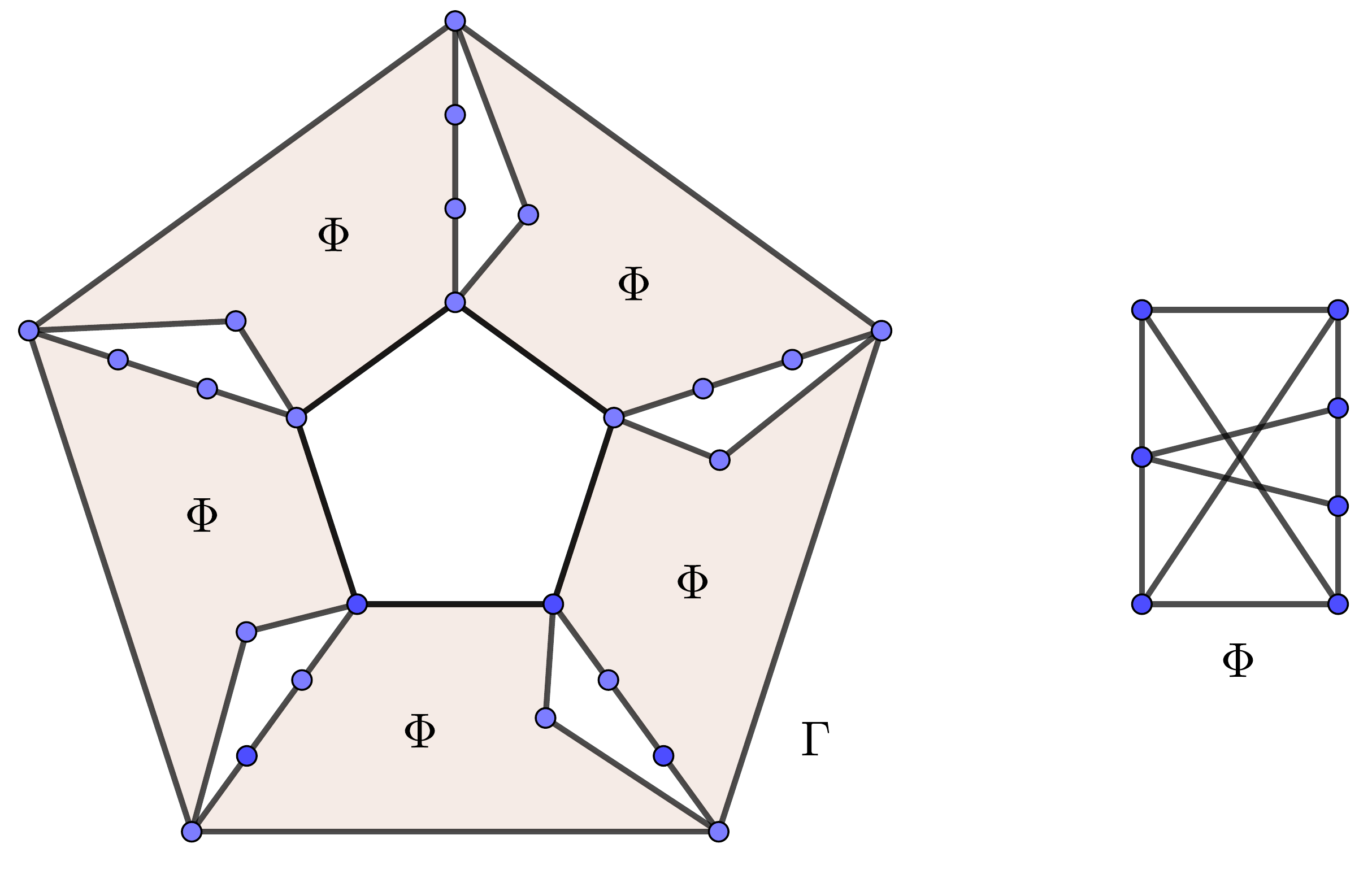}
\caption{A graph $\Gamma$ such that $\mathbb{E}(\Gamma)$ is hyperbolic but not a quasi-tree.}
\label{NotQT}
\end{center}
\end{figure}
\end{ex}

\begin{ex}
In some sense, minsquare graphs (i.e., graphs which do not contain proper minsquare subgraphs) are constructed from squares, so it is natural to compare the family of minsquare graphs with the family of $\mathcal{CFS}$ graphs introduced in \cite{RACGdiv} (characterising right-angled Coxeter groups with quadratic divergence). Recall that a graph $\Gamma$ is $\mathcal{CFS}$ if there exists a sequence of induced squares $C_1, \ldots, C_n$ covering all the vertices of $\Gamma$ such that, for every $1 \leq i \leq n-1$, the $C_i \cap C_{i+1}$ contains two non-adjacent vertices. Figure \ref{NotCFS} gives two examples of graphs, the first one being minsquare but not $\mathcal{CFS}$ and the second one $\mathcal{CFS}$ but not minsquare. Consequently, none of our two families of graphs contains the other.
\begin{figure}
\begin{center}
\includegraphics[scale=0.4]{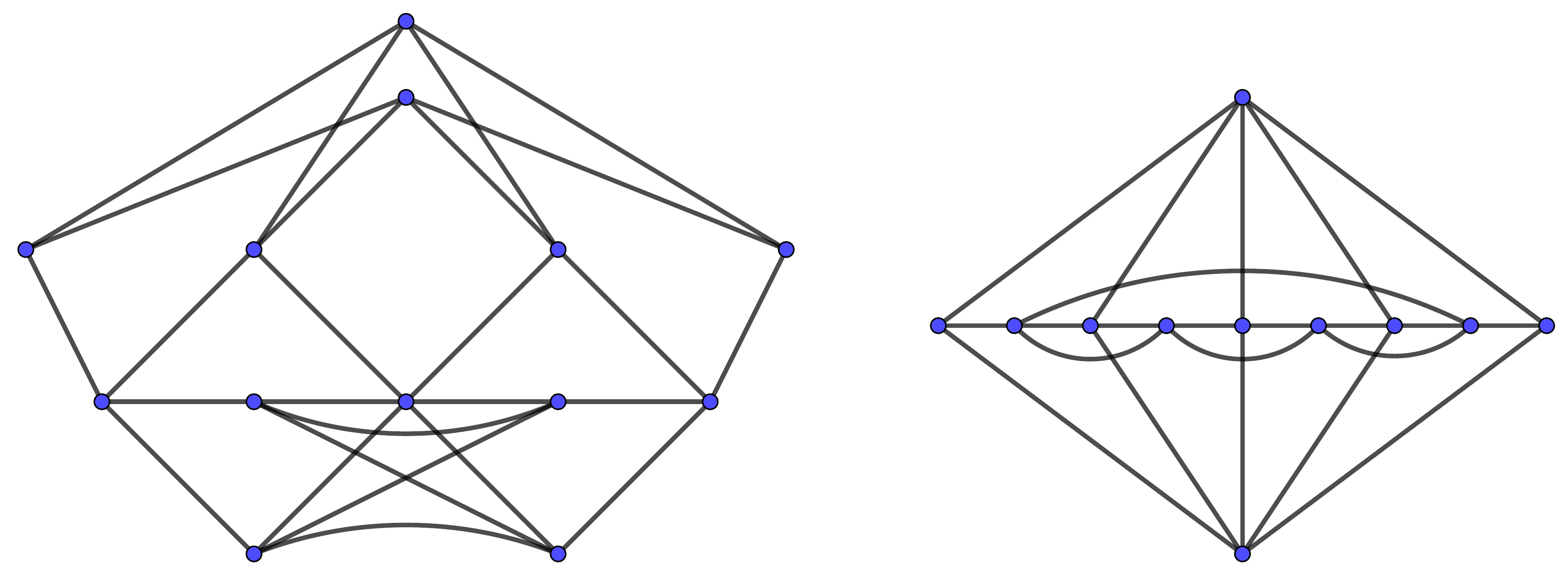}
\caption{A minsquare graph which is not $\mathcal{CFS}$; and a $\mathcal{CFS}$ graph which is not minsquare.}
\label{NotCFS}
\end{center}
\end{figure}
\end{ex}

\begin{ex}\label{ex:PhiPsi}
Let $\Gamma$ be a finite simplicial graph. Assume that $\Gamma$ is the union of two minsquare subgraphs $\Gamma_1$ and $\Gamma_2$. As a consequence, the right-angled Coxeter group can be thought of as constructed from rigid pieces isometric to $\langle \Gamma_1 \rangle$ and $\langle \Gamma_2 \rangle$ whose organisation is encoded by the electrification $\mathbb{E}(\Gamma)$. Do the data of $\langle \Gamma_1 \rangle$, $\langle \Gamma_2 \rangle$, $\mathbb{E}(\Gamma)$ up to quasi-isometry determine the quasi-isometry class of $C(\Gamma)$? Our example below shows that it may not be the case.

\medskip \noindent
Fix a graph $\Gamma$ containing two vertices $a,b$ at distance three apart and such that any two induced squares are connected by a sequence of induced squares such that two consecutive squares share two opposite vertices. See for instance Figure \ref{PhiPsi}. Given two copies $\Gamma_1,\Gamma_2$ of $\Gamma$, let $\Phi$ denote the graph obtained from $\Gamma_1$ and $\Gamma_2$ by identifying $a,b \in V(\Gamma_1)$ respectively with $a,b \in V(\Gamma_2)$; and let $\Psi$ denote the graph obtained from $\Gamma_1$ and $\Gamma_2$ by identifying $a \in V(\Gamma_1)$ with $a \in V(\Gamma_2)$. By construction, the two right-angled Coxeter groups $C(\Phi)$ and $C(\Psi)$ have the same eccentric subgroups (up to isomorphism). We claim that the electrifications $\mathbb{E}(\Phi)$ and $\mathbb{E}(\Psi)$ are quasi-isometric.

\medskip \noindent
Notice that $C(\Phi)$ decomposes as an amalgamated product $\langle \Gamma_1 \rangle \underset{\langle a,b \rangle}{\ast} \langle \Gamma_2 \rangle$, and $C(\Psi)$ as $\langle \Gamma_1 \rangle \underset{\langle a \rangle}{\ast} \langle \Gamma_2 \rangle$. A general fact is that the Cayley graph of an amalgamated product $A \underset{C}{\ast} B$ with respect to the generating set $A \cup B$ is quasi-isometric to the associated Bass-Serre tree (see for instance the proof of \cite[Theorem 1.3]{OsinWeakRH}). Consequently, $\mathbb{E}(\Phi)$ and $\mathbb{E}(\Psi)$ are quasi-isometric to the Bass-Serre trees $T_1$ and $T_2$ corresponding to the two amalgamated products above. But $T_1$ and $T_2$ are two simplicial trees all of whose vertices have infinite (countable) degree, so they must be isomorphic. Therefore, $\mathbb{E}(\Phi)$ and $\mathbb{E}(\Psi)$ must be quasi-isometric, as claimed.

\medskip \noindent
However, $C(\Phi)$ and $C(\Psi)$ are not quasi-isometric. Indeed, as a consequence of \cite[Theorems 8.7.2 and 8.7.4]{DavisCoxeter},  $C(\Phi)$ has just one end but $C(\Psi)$ has infinitely many ends.

\begin{figure}
\begin{center}
\includegraphics[scale=0.4]{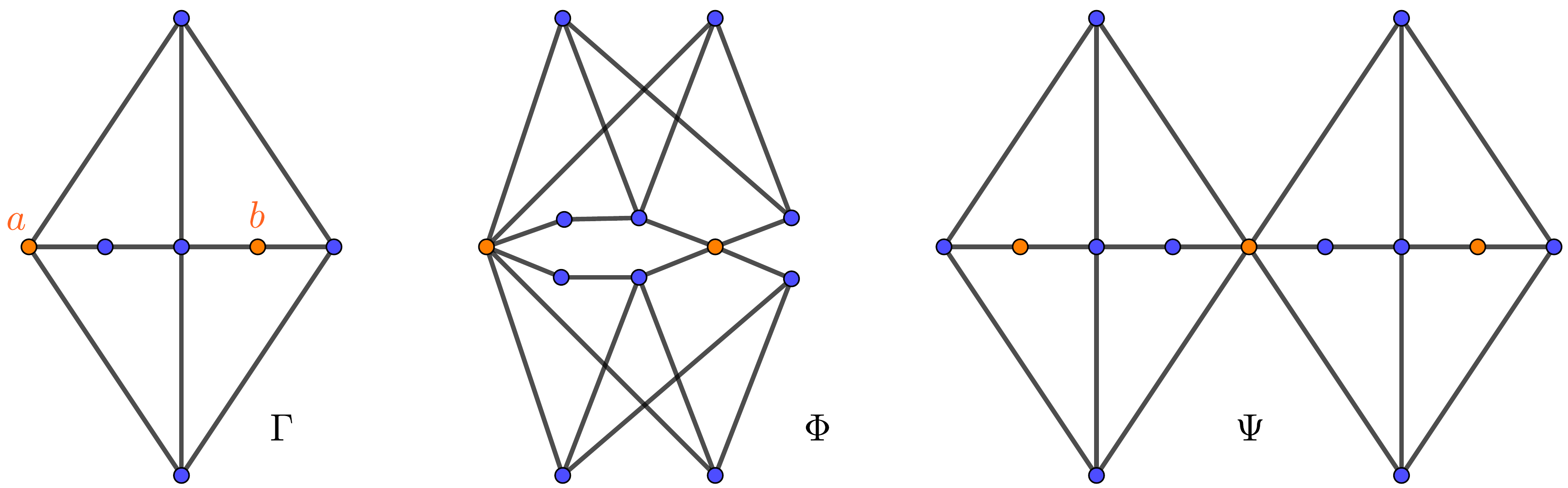}
\caption{Possible graphs $\Gamma$, $\Phi$ and $\Psi$ from Example \ref{ex:PhiPsi}.}
\label{PhiPsi}
\end{center}
\end{figure}
\end{ex}

\section{Open questions}\label{section:questions}

\noindent
A natural but probably very difficult problem regarding Theorem \ref{thm:IntroMain} is:

\begin{problem}
Classify up to quasi-isometry the right-angled Coxeter groups / graph products of finite groups defined by minsquare graphs.
\end{problem}


\noindent
Next, many interesting questions can be asked about the geometry of the electrification. For instance, given a finite simplicial graph $\Gamma$ and a collection of finite groups $\mathcal{G}$ indexed by $V(\Gamma)$:

\begin{question}
How many ends has $\mathbb{E}(\Gamma, \mathcal{G})$?
\end{question}


\begin{question}
When is $\mathbb{E}(\Gamma, \mathcal{G})$ bounded? Does it happen if and only if $\Gamma$ decomposes as the join of a minsquare subgraph and a complete graph?
\end{question}

\begin{question}
When is $\mathbb{E}(\Gamma, \mathcal{G})$ a quasi-line? Does it happen if and only if $\Gamma \mathcal{G}$ is virtually cyclic?
\end{question}

\begin{question}
When is $\mathbb{E}(\Gamma, \mathcal{G})$ a quasi-tree?
\end{question}

\noindent
In view of Theorem \ref{thm:MorseAllHyp}, a natural question to ask is:

\begin{question}\label{question:MorseFree}
When are all infinite-index Morse subgroups of a graph product of finite groups virtually free?
\end{question}

\noindent
(Notice that Remark \ref{remark:MorseRAAG} allows us to construct many examples of such graph products. However, these examples are very specific.)

\medskip \noindent
As mentioned in the introduction, it is known that freely irreducible right-angled Artin groups have all their infinite-index Morse subgroups free. Loosing speaking, they have only few Morse subgroups. The situation seems quite different for graph products of finite groups (or even for right-angled Coxeter). Motivated by this observation, we ask:

\begin{question}
Let $G$ be a (cocompact) special group. Does there exist a graph products of finite groups $H$ such that $G$ embeds into $H$ as a Morse subgroup?
\end{question}

\noindent
It is well-known that a special group always embeds into a right-angled Coxeter group \cite{SpecialGroups}, but the image of the embedding is in general far from being a Morse subspace. 

\medskip \noindent
Finally, it would be interesting to find other families of groups where the strategy of our article also applies.

\begin{question}
Do there exist other families of groups all of whose eccentric subspaces are at finite Hausdorff distance from subgroups?
\end{question}

\addcontentsline{toc}{section}{References}

\bibliographystyle{alpha}
{\footnotesize\bibliography{EccentricRACG}}

\end{document}